\documentclass[11pt]{article}
\usepackage{amsmath}
\usepackage[utf8]{inputenc}
\usepackage{amsthm}
\usepackage{thmtools}
\usepackage{hyperref, cleveref}
\newtheoremstyle{theorem}{4mm}{1mm}{\itshape}{ }{\bfseries}{.}{ }{}
\theoremstyle{theorem}
\newtheorem{theoremA}{Theorem A\ignorespaces}
\usepackage[a4paper, total={7in, 9in}]{geometry}
\usepackage{blindtext}
\usepackage{comment}
\usepackage{geometry}
\providecommand{\keywords}[1]
{
	\small	
	\textbf{\textit{Keywords:}} #1
}
\allowdisplaybreaks
\numberwithin{equation}{section}
\newtheorem{mytheorem}{Theorem}

\newtheorem{pro}{Proposition}
\numberwithin{pro}{section}
\newtheorem{lemma}{Lemma}
\numberwithin{lemma}{section}
\newtheorem{remark}{Remark}
\numberwithin{remark}{section}
\numberwithin{mytheorem}{section}
\numberwithin{corollary}{section}
\usepackage{amsbsy}
\usepackage{graphicx}
\usepackage[T1]{fontenc}
\usepackage{lmodern}
\usepackage{imakeidx}
\usepackage{amssymb}
\usepackage{thmtools}
\usepackage{cite}
\newenvironment{myproof}[2] {\paragraph{Proof of {#1} {#2} :}}{\hfill$\square$}

\usepackage[english]{babel}

\newtheorem{defi}{Definition}[section]
\usepackage[usenames,dvipsnames]{color}
\title{ Normalized solutions to a Choquard equation involving mixed
	local and nonlocal operators}
\author{J. Giacomoni\footnote{LMAP (UMR E2S UPPA CNRS 5142) Bat. IPRA, Avenue de l'Universit\'{e}, 64013 Pau, France. e-mail: jacques.giacomoni@univ-pau.fr}\;, Nidhi Nidhi\footnote{Department of Mathematics, Indian Institute of Technology, Delhi, Hauz Khas, New Delhi-110016, India. e-mail: nidhi.nidhi@maths.iitd.ac.in}\; and K. Sreenadh\footnote{Department of Mathematics, Indian Institute of Technology, Delhi, Hauz Khas, New Delhi-110016, India. e-mail: sreenath@maths.iitd.ac.in}\;}
\date{}
\begin{document}
	\maketitle
	\begin{abstract}
		\noindent In the present paper, we study the existence of normalized solutions for a  Choquard type equation involving mixed diffusion type operators.  We also provide regularity results of these solutions. Next, the equivalence between existence of normalized  solutions and the existence of  normalized ground states is established.\\
		\noindent \keywords{Normalized solutions, local-nonlocal mixed operator, Choquard equation, Sobolev regularity, existence results, ground states.}
	\end{abstract}
	\section{Introduction and Main Results}
	\noindent In the present work, we are interested to study the existence and properties of the  normalized solutions to the following Choquard equation involving mixed local and nonlocal type operators:
	\begin{equation}\label{prob}
		\begin{array}{rcl}
			\mathcal{L}	u+u & = & \mu (I_{\alpha}*|u|^p)|u|^{p-2}u\;\;\text{in } \mathbb{R}^n,\\
			\left\| u \right\|_2^2 & = & \tau,
		\end{array}
	\end{equation}
	where $\frac{n+\alpha}{n}\leq p \leq \frac{2s+n+\alpha}{n}$, $\tau >0$ is a constant, $\mu>0$ is a parameter, $I_{\alpha}$ is the Riesz potential of order $\alpha \in (0,n)$ defined by
	$$I_{\alpha}=\frac{A_{n,\alpha}}{|x|^{n-\alpha}},\;\;\text{with } A_{n,\alpha}=\frac{\Gamma(\frac{n-\alpha}{2})}{\pi^{\frac{n}{2}}2^{\alpha}\Gamma(\frac{\alpha}{2})}\; \text{  for every } x\in \mathbb{R}^n \setminus \{0\}$$ 
	and the mixed operator $\mathcal{L}$ is given by
	$$\mathcal{L}=-\Delta+\lambda(-\Delta)^s\;\text{ for some } s\in (0,1) \text{ and parameter } \lambda>0.$$
	Here, the operator $\mathcal{L}$ is referred as a mixed operator since it contains both local and nonlocal properties. It involves the classical Laplacian $(-\Delta)$ and the fractional Laplacian $(-\Delta)^s$ defined as:
	$$(-\Delta)^su(x)=C(n,s)\text{P.V.}\int_{\mathbb{R}^n}\frac{u(x)-u(y)}{|x-y|^{n+2s}}dy\;\text{ for } s\in(0,1),$$
	where $C(n,s)$ is a normalizing constant given by
	$$C(n,s)=\left(\int_{\mathbb{R}^n}\frac{1-cos(x)}{|x|^{n+2s}}dx\right)^{-1}$$
	and P.V. is the abbreviation for principal value. 
	
	\noindent We are interested to study the operator $\mathcal{L}$ due to its wide range of applications. Broadly speaking, such an operator comes into the picture whenever the impact on a physical phenomenon is caused by some local as well as nonlocal changes. One may come across this operator while studying the bi-model power law distribution processes, see \cite{pagnini2021should}.
	It is also considered in the theory of optimal searching, biomathematics and animal foraging, see \cite{dipierro2022non} for further extend.
	Issues regarding the existence of solutions, their regularity 
	and symmetry properties, Faber-Krahn type inequality, Neumann problems and Green functions estimates have been investigated in different contributions (see in particular \cite{biagi2021global,biagi2022mixed,abatangelo2021elliptic}) and for a further study about  the semilinear elliptic equations involving mixed operators, we quote \cite{arora2021combined} and references therein. In the present work, we would like to study a certain type of solutions, precisely normalized solutions, for a class of elliptic problems involving the operator $\mathcal{L}$ that to the best of our knowledge has not been discussed in the former literature.
	
	\noindent We highlight that the study of normalized solutions has significant importance in physics, as it represents an entity satisfying the dynamics along with a fixed mass. Consider for instance the following semilinear equation with a constraint:
	\begin{equation}\label{Norm_sol}
		\left\{ \begin{array}{rl}   	
			& 	-\Delta u  =  \lambda u +g(u)\;\;\text{in } \mathbb{R}^n,\\
			&  	\left\| u \right\|_2^2 =  c.
		\end{array}
		\right.
	\end{equation} 
	The physical motivation to study \eqref{Norm_sol} can be given by the fact that its solution gives stationary states of a nonlinear Schr$\ddot{\text{o}}$dinger equation with prescribed $L^2-$norm. Formally, the solution of \eqref{Norm_sol} is called a normalized solution. In \cite{jeanjean1997existence}, Jeanjean obtained the existence of radial solutions for \eqref{Norm_sol} under some assumptions on $g$. In \cite{bartsch2012normalized}, the existence of infinite solutions to \eqref{Norm_sol} under same assumptions has been shown.
	Further in \cite{noris2015existence}, the normalized solutions are discussed and described in case of  bounded domains with Dirichlet boundary conditions. Considering the domain to be the unit ball and $g(x)=|x|^{p-1}x$, the existence of normalized solutions has been seen for $p$ lying in $(1,1+\frac{4}{n}),\; (1+\frac{4}{n}, 2^*-1)$ and $p=1+\frac{4}{n}$ under some conditions on $c$. 
	Moreover, the problem in general bounded domains has been dealt by the authors in \cite{pierotti2017normalized}.  Existence of normalized solutions of nonlinear Schr$\ddot{\text{o}}$dinger systems has been also extensively studied, interested readers can go through      \cite{gou2018multiple,bartsch2016normalized,bartsch2018normalized,bartsch2019multiple,noris2014stable,noris2019normalized}. Normalized solutions are also taken into consideration in the study of quadratic ergodic mean field games system, see in particular \cite{pellacci2021normalized}.
	
	\noindent	Recently, the study of normalized solutions of the Choquard equation has attracted many researchers. This investigation started with the work of Lieb in \cite{lieb1977existence}, where they studied the existence and uniqueness of the minimizing solution for the problem:
	\begin{equation}\label{Norm_Choq}
		\left\{ \begin{array}{rcl}   	
			-\Delta u  +  \lambda u & = & \mu (I_\alpha*|u|^p)|u|^{p-2}u\;\;\text{in } \mathbb{R}^n,\\
			\left\| u \right\|_2^2 & =  &\tau.
		\end{array}
		\right.
	\end{equation} 
	Pekar's work in \cite{pekar1954untersuchungen}, which examines the model in quantum theory of a polaron at rest by considering the equation 
	\begin{equation}\label{Choq}
		\left\{ \begin{array}{rl}   	
			& 	-\Delta u  +  u = (I_2*|u|^2)u\;\;\text{in } \mathbb{R}^3,\\
			&  	u \in H^1(\mathbb{R}^3),
		\end{array}
		\right.
	\end{equation} 
	shed light on the significance of the Choquard equation. Further in 1976, Choquard used the energy functional associated to \eqref{Choq} in order to study a suitable approximation to Hartree-Fock theory for one component plasma. The equation is mainly used to describe an electron trapped in its own hole (see \cite{lieb1977existence}). It also arises in quantum mechanics, see \cite{penrose1996gravity} and references therein. Many authors have subsequently studied the existence, the multiplicity and qualitative properties of the solution to the problem \eqref{Norm_Choq}, see  \cite{filippucci2020singular,moroz2013groundstates, liu2022another}. 
	
	\noindent Considering \eqref{prob}, we look for weak solutions in the energy space $H^1(\mathbb{R}^n)$ equipped with the inner product defined as:
	$$\langle u,v \rangle:=\int_{\mathbb{R}^n}\nabla u.\nabla vdx+\int_{\mathbb{R}^n}u.vdx + \lambda \ll u,v \gg,$$
	where $$\ll u,v \gg:=\frac{C(n,s)}{2}\int_{\mathbb{R}^n}\int_{\mathbb{R}^n}\frac{(u(x)-u(y)).(v(x)-v(y))}{|x-y|^{n+2s}}dxdy,$$
	for each $u,v \in H^1(\mathbb{R}^n)$ and the corresponding norm as follows:
	$$\left\| u \right\|_{H^1(\mathbb{R}^n)}^2:=\left\| \nabla u \right\|_2^2+\lambda [u]^2+\left\| u \right\|_2^2; \text{ with }[u]^2:=\frac{C(n,s)}{2}\int_{\mathbb{R}^n}\int_{\mathbb{R}^n}\frac{|u(x)-u(y)|^2}{|x-y|^{n+2s}}dxdy.$$
 Similarly, $\|u\|_{H^s(\mathbb{R}^n)}^2:=[u]^2+\|u\|^2_2.$
	
	\noindent Precisely, we give below the notion of weak solutions:
	\begin{defi}
		A function $u\in H^1(\mathbb{R}^n)$ is said to be a weak solution of \eqref{prob} if $\left\|u\right\|_2^2=\tau$ and
		\begin{equation*}
			\int_{\mathbb{R}^n}\nabla u.\nabla v dx+\lambda \ll u,v \gg + \int_{\mathbb{R}^n}u.v dx
			=\mu\int_{\mathbb{R}^n}(I_{\alpha}*|u|^p)|u|^{p-2}uvdx,
		\end{equation*}
		for each $v \in H^1(\mathbb{R}^n)$.
	\end{defi}
	\noindent The energy functional associated to \eqref{prob} is given for any $u\in H^1(\mathbb{R}^n)$ by :
	$$S_\lambda(u)  =  \frac{1}{2}\left\| \nabla u \right\|_2^2+\frac{1}{2}\left\| u \right\|_2^2-\frac{\mu}{2p}\int_{\mathbb{R}^n}(I_{\alpha}*|u|^p)|u|^p +\frac{\lambda [u]^2}{2}.$$
	First, we investigate the regularity of weak solutions. More precisely, we establish the following:
	\begin{mytheorem}\label{thm 1.1}
		Assume $n\geq 3$ and $\alpha\in(0,n)$. If $u$ is a solution of \eqref{prob}, then 
		$u\in L^q(\mathbb{R}^n)$ for every $q\in[2,\frac{n}{\alpha}\frac{2n}{n-2})$.
	\end{mytheorem}
	\noindent	Using Theorem 1.1 we obtain Sobolev regularity of weak solutions:
	\begin{mytheorem}\label{thm1.2}
		Let $n\geq 3$, $n-\alpha<4$, $2\leq p \leq \frac{n+\alpha}{n-2}$ and $u\in H^1(\mathbb{R}^n)$ is a solution of \eqref{prob}, then $u\in W^{2,q}_{loc}(\mathbb{R}^n)$ for all $q\geq 1$.
	\end{mytheorem}
	\noindent An important tool to establish the above regularity results is the following well-known Hardy-Littlewood-Sobolev inequality (see \cite{Moroz2017guide}):
	
	\begin{pro}\label{prop1.1}
		Let $t,r>1$ and $0<\alpha <n$ with $1/t+1/r=1+\alpha/n$, $f\in L^t(\mathbb{R}^n)$ and $h\in L^r(\mathbb{R}^n)$. There exists a sharp constant $C(t,\alpha,n)$ independent of $f,h$, such that
		\begin{equation}\label{co9}
			\int_{\mathbb{R}^n}\int_{\mathbb{R}^n}\frac{f(x)h(y)}{|x-y|^{n-\alpha}}~dxdy \leq C(t,\alpha,n) \|f\|_{L^t}\|h\|_{L^r}.
		\end{equation}
		If $t=r=2n/(n+\alpha)$, then
		\begin{align}\label{C_alpha}
			C(t,\alpha,n)=C(n,\alpha)= \pi^{\frac{n-\alpha}{2}}\frac{\Gamma(\frac{\alpha}{2})}{\Gamma(\frac{n+\alpha}{2})}\left\lbrace \frac{\Gamma(\frac{n}{2})}{\Gamma(n)}\right\rbrace^{-\frac{\alpha}{n}}.
		\end{align}
		Equality holds in  \eqref{co9} if and only if $\frac{f}{h}\equiv constant$ and
		$\displaystyle h(x)= A(\gamma^2+|x-a|^2)^{(N+\alpha)/2}$
		for some $A\in \mathbb{C}, 0\neq \gamma \in \mathbb{R}$ and $a \in \mathbb{R}^n$.
	\end{pro}
	\noindent From this inequality, it follows that 
	\begin{align*}
		{\mathcal{A}_q(u):=	\int_{\mathbb R^n}\int_{\mathbb R^n}\frac{|u(x)|^{q}|u(y)|^{q}}{|x-y|^{n-\alpha}}~dxdy}
	\end{align*}
	is well defined if $\frac{n+\alpha}{n}\leq q \leq \frac{n+\alpha}{n-2}=2^*_{\alpha} $. The exponent $q = 2^*_{\alpha}$ is known as Hardy-Littlewood-Sobolev critical exponent and similarly to the usual critical exponent, $H^1_0(\Omega)\ni\, u\to\, \mathcal{A}_{2^*_\alpha}(u)$ is continuous for the norm topology but not for the weak topology. \\
	The existence issue has been discussed for problems involving local operators.  In particular, in \cite{lei2023sufficient},  Lei, Yang and Zhang have proved the equivalence of normalized solutions and ground state solutions to the problem:
	\begin{equation}\label{Norm_Choq_suff&ness}
		\left\{ \begin{array}{rl}   	
			-\Delta u  + u& =\mu (I_\alpha*|u|^p)|u|^{p-2}u\;\;\text{in } \mathbb{R}^n,\\
			\left\| u \right\|_2^2 & =  \tau,
		\end{array}
		\right.
	\end{equation}   
	for the case $\frac{n+\alpha}{n}<p<\frac{2+n+\alpha}{n}$ and hence obtained the existence of weak solutions to \eqref{Norm_Choq_suff&ness}. Further they deduced nonexistence results for $p=\frac{n+\alpha}{n}$, $p=\frac{n+\alpha}{n-2}$ and $p=\frac{2+n+\alpha}{n}$. By a ground state solution, we mean here a solution to the constrained minimization problem given by:
	\begin{equation}\label{groundState}
		\left\{ \begin{array}{rl}   	
			&  \left\| u \right\|_2^2 =\tau,   \\
			&  F(u) = \min\{F(w)\;:\; w\in H^1(\mathbb{R}^n) \text{ such that } \left\| u \right\|_2^2=\tau \}, 
		\end{array}
		\right.
	\end{equation}
	where $F$ is the energy functional associated to \eqref{Norm_Choq_suff&ness}. We extend similar results for our problem involving mixed type operators.
	The following notations will be used through out this paper:
	$$H(u):=\left\|u\right\|_2^2,\;\;T(u) := \left\|\nabla u\right\|_2^2+\lambda[u]^2,\;\;\Gamma := \{u\in H^1(\mathbb{R}^n)\;:\;H(u)=\tau\},$$
	$$A:=\{u\in H^1(\mathbb{R}^n)\;:\; u\neq0\text{ and }S_{\lambda}'(u)(u)=0\},\;\;	G:=\{u\in A\;:\;S_{\lambda}(u)\leq S_{\lambda}(v)\text{ for all }v\in A\}. $$
	Results corresponding to the problem \eqref{prob} are as follows:
	\begin{mytheorem}\label{Theorem 1.3}
		For $\frac{n+\alpha}{n}<p<\frac{2s+n+\alpha}{n}$, $s\in(0,1)$, $\mu>0$ and $\lambda>0$ being sufficiently small, the problem \eqref{prob} has a weak solution.
	\end{mytheorem} 
	\noindent We investigate the existence of ground state solutions by considering the following minimization problem:
	\begin{equation}\label{minS}
		\left\{ \begin{array}{rl}   	
			&  u \in \Gamma,   \\
			&  S_{\lambda}(u) = \min\{S_{\lambda}(w)\;:\; w\in \Gamma \}.
		\end{array}
		\right.
	\end{equation}
	Then we prove the following result.
	\begin{mytheorem}\label{Theorem 1.4}
		For $\frac{n+\alpha}{n}<p<\frac{2s+n+\alpha}{n}$, $s\in(0,1)$, $\mu>0$ and $\lambda>0$ being sufficiently small, the problems \eqref{prob} and \eqref{minS} are equivalent.
	\end{mytheorem}
	\noindent We believe that theorems \ref{Theorem 1.3} and \ref{Theorem 1.4} hold for any $\lambda>0$ and $\frac{n+\alpha}{n}<p<\frac{n+\alpha}{n-2}$. In this regard, we deal with nonexistence results and show that as $p$ takes the upper critical ($p=\frac{n+\alpha}{n-2}$) or lower critical value ($p=\frac{n+\alpha}{n}$) with respect to the above mentioned Hardy-Littlewood-Sobolev inequality, the problem \eqref{prob} has no solution. Precisely, we establish:
	\begin{mytheorem}\label{Theorem 1.5}
		Assume that $p=\frac{n+\alpha}{n}$ or $p=\frac{n+\alpha}{n-2}$, then the problem \eqref{prob} has no weak solution in $ H^1(\mathbb{R}^n)$.
	\end{mytheorem}
	
	\noindent The paper is organized as follows. Section 2 deals with the proof of regularity results (Theorems \ref{thm 1.1} and \ref{thm1.2}) whereas Section 3 is devoted to the existence of weak solutions and the equivalence result given in Theorems \ref{Theorem 1.3} and \ref{Theorem 1.4} respectively.  Concerning the regularity of solutions, we obtain that a weak solution lies in $L^q(\mathbb{R}^n)$ for all $q \in [2,(\frac{n}{\alpha})\frac{2n}{n-2})$ by following the ideas in \cite[Moroz-Schaftingen]{moroz2015existence}. Further, using the estimate given in Proposition 2.2 and Theorem 1.1, we obtain the Sobolev regularity given in Theorem \ref{thm1.2}. Next, we get the existence result by appealing an "asymptotic" Pohozaev type identity together with concentration compactness arguments applied to  a minimizing sequence and analysing the associated sequence of Lagrange multipliers $\{\Lambda^\lambda_k\}$ (see Lemma \ref{lemma3.2}). Furthermore, Theorem \ref{Theorem 1.4} is established by showing that a function is a ground-state solution if and only if it belongs to $G\cap \Gamma$. Finally the proof of nonexistence result (Theorem \ref{Theorem 1.5}) ends Section 3. \noindent The Pohozaev identity serves as a crucial instrument to get these existence and non existence results, it can obtained by combining the work of Moroz-Schaftingen \cite[proposition~3.1]{moroz2013groundstates} and Anthal-Garain \cite[Theorem~2.5]{Anthal2025Pohozaev} (where the authors use difference quotient approximations as in \cite{ambrosio} and \cite{jiaquanliu}). We have added the sketch of the proof in the appendix. We point out that the presence of the non local part of the mixed operator introduces technical difficulties and forces the use of accurate estimates along the proofs of our main results. \\
	\section{Regularity Results }
	In this section, we will prove regularity results for a solution to \eqref{prob}. For this, we first show that every solution of \eqref{prob} belongs to $L^q(\mathbb{R}^n)$ for all $q \in [2,\frac{n}{\alpha}\frac{2n}{n-2})$ using the following proposition.
	\begin{pro}
		If $M,N \in L^{\frac{2n}{\alpha}}(\mathbb{R}^n)+L^{\frac{2n}{\alpha+2}}(\mathbb{R}^n)$ and $u \in H^1(\mathbb{R}^n)$ solves
		\begin{equation}\label{reg_prob}
			-\Delta u +\lambda(-\Delta)^su+u = \mu(I_{\alpha}*Mu)N\;\;\text{ in }\mathbb{R}^n,
		\end{equation}
		then $u\in L^q(\mathbb{R}^n)$ for every $q\in[2,\frac{n}{\alpha}\frac{2n}{n-2})$.
	\end{pro}
	\begin{proof}
		By Lemma 3.2 of \cite{moroz2015existence}, taking $\theta=1$, for $\epsilon=\frac{1}{\sqrt{8\mu}}$ there exists $\beta>0$ such that 
		\begin{eqnarray*}
			4\mu\int_{\mathbb{R}^n}(I_{\alpha}*|M||w|)|N||w| & \leq & \frac{\left\| \nabla w\right\|_2^2}{2}+\frac{\beta \left\| w \right\|_2^2}{2}
			\leq \frac{\left\| \nabla w \right\|_2^2}{2}+\frac{\beta \left\| w \right\|_2^2}{2}
			+\frac{\lambda}{2}[w]^2,
		\end{eqnarray*}
		for every $w\in H^1(\mathbb{R}^n)$. Motivated by the poineering work of V. Moroz and J. Van Schaftingen in \cite{moroz2015existence} (see in particular Proposition 3.1 there), taking $M=M_1+M_2$ and $N=N_1+N_2$ where $M_1$, $N_1\in L^{\frac{2n}{\alpha}}(\mathbb{R}^n)$ and $M_2$, $N_2\in L^{\frac{2n}{\alpha+2}}(\mathbb{R}^n)$, we define the sequences $\{M_k\}$ and $\{N_k\}$ as follows:
		$$M_k=M_{1,k}+ M_{2,k} \text{ and } N_k=N_{1,k}+N_{2,k},$$
		where $M_{1,k}=M_1$, $N_{1,k}=N_1$, $M_{2,k}=\chi_{\bar{B}(0,k)}\bar{M}_{2,k}$ and $N_{2,k}=\chi_{\bar{B}(0,k)}\bar{N}_{2,k}$ with $\chi_S$ being the characterisitic function on set $S$ and $\bar{M}_{2,k}$, $\bar{N}_{2,k}$ are the truncations of $M_2$ and $N_2$ respectively, defined as:
		$$\bar{M}_{2,k}=\max\{-k,\min\{k,M_2\}\} \text{ and } \bar{N}= \max\{-k,\min\{k,N_2\}\}.$$
		Then clearly
		$\{M_k\}_{k\in \mathbb{N}}$ and $\{N_k\}_{k\in \mathbb{N}}$ belong to $L^{\frac{2n}{\alpha}}(\mathbb{R}^n)$ and satisfy 
		$|M_k|\leq  2|M|,\; \{M_k(x)\}\rightarrow M(x) $ almost everywhere,
		$|N_k|\leq 2|N|$ and $ \{N_k(x)\}\rightarrow N(x) $ almost everywhere  in $\mathbb{R}^n$.
		For each $k\in \mathbb{N}$, we define $a_k:H^1(\mathbb{R}^n)\times H^1(\mathbb{R}^n)\rightarrow \mathbb{R}$ as:
		\begin{equation*}
			a_k(w,v) :=  \int_{\mathbb{R}^n}\nabla w.\nabla vdx+\beta\int_{\mathbb{R}^n}w.vdx-\mu\int_{\mathbb{R}^n}(I_{\alpha}*M_kw)N_kvdx+\lambda\ll w,v\gg.
		\end{equation*}
		Since $M_k,N_k\in L^{\frac{2n}{\alpha}}(\mathbb{R}^n)\subset L^{\frac{2n}{\alpha}}(\mathbb{R}^n)+L^{\frac{2n}{\alpha+2}}(\mathbb{R}^n)$ for every $k\in \mathbb{N}$,
		\begin{eqnarray*}
			a_k(v,v)
			& \geq &\left\| \nabla v \right\|_2^2+\beta\left\| v \right\|_2^2 +\lambda[v]^2
			-\left(\frac{\left\| \nabla v \right\|_2^2}{2}+\frac{\beta\left\| v \right\|_2^2}{2}+\frac{\lambda }{2}[w]^2\right)\\
			&\geq & \text{C}\left(\left\| \nabla v \right\|_2^2+\left\| w \right\|_2^2+\lambda [w]^2\right),\text{ here C}=\min\{\frac{1}{2},\frac{\beta}{2}\}.
		\end{eqnarray*}
		Thus $a_k$ is a coercive bilinear form and hence applying Lax-Milgram theorem \cite[corollary~5.8]{haim2011functional} with the linear continuous form $f:H^1(\mathbb{R}^n)\rightarrow\mathbb{R}$ defined as:
		$$f(v)=\int_{\mathbb{R}^n}(\beta-1)uvdx,\;\;\forall v \in H^1(\mathbb{R}^n),$$
		there exists a unique $u_k\in H^1(\mathbb{R}^n)$ such that $a_k(u_k,v)=f(v)$ for every $v \in H^1(\mathbb{R}^n)$. This implies that $u_k$ is a solution of
		\begin{equation}\label{seq_prob}
			-\Delta u_k+\lambda(-\Delta)^su_k+\beta u_k= \mu(I_{\alpha}*M_ku_k)N_k+(\beta-1)u\;\;\text{ in }\mathbb{R}^n.
		\end{equation}
		Now, taking $\phi\in H^{1}(\mathbb{R}^n)$, since $N_{1,k}\phi= N_1\phi\in L^{\frac{2n}{n+\alpha}}(\mathbb{R}^n)$ for all $k\in \mathbb{N}$, we get $(I_\alpha*(N_{1,k}\phi))\rightarrow (I_\alpha*(N_{1}\phi))$ in $L^{\frac{2n}{n-\alpha}}(\mathbb{R}^n)$, by \cite[A-2]{moroz2017guide}. Similarly, since, by Dominated convergence theorem, $\{N_{2,k}\phi\}\rightarrow N_2\phi$ in $L^{\frac{2n}{n+\alpha}}(\mathbb{R}^n)$, we get $(I_\alpha*(N_{2,k}\phi))\rightarrow (I_\alpha*(N_{2}\phi))$ in $L^{\frac{2n}{n-\alpha}}(\mathbb{R}^n)$. Thus
		$$I_{\alpha}*N_k\phi \rightarrow I_{\alpha}*N\phi \text{ in } L^{\frac{2n}{n-\alpha}}(\mathbb{R}^n).$$
		Now,  from the uniform coercivity of $a_k$, taking $u_k $ as a test function in \eqref{seq_prob} and using Cauchy-Schwarz inequality, we obtain the boundedness of $\{u_k\}$ in $H^1(\mathbb{R}^n)$, hence $\{u_k\}$ converges weakly to some $w\in H^1(\mathbb{R}^n)$ up to a subsequence. However, one can easily check that $\{M_ku_kg\}\rightarrow Mwg$ for every $g\in L^{\frac{2n}{n-\alpha}}(\mathbb{R}^n)$ and hence:
		$$\int_{\mathbb{R}^n}(I_{\alpha}*M_ku_k)N_k\phi= \int_{\mathbb{R}^n}(I_{\alpha}*N_k\phi)M_ku_k\rightarrow \int_{\mathbb{R}^n}(I_{\alpha}*Mw)N\phi.$$
		This indicates that $w$ satisfies $-\Delta w+\lambda(-\Delta)^sw+\beta w= \mu(I_{\alpha}*Mw)N+(\beta-1)u \text{ in }H^{-1}(\mathbb{R}^n)$ and since $u$ is its unique solution of \eqref{reg_prob} (by the coercivity of the associated bilinear form), $w=u$ (again by coercivity of the bilinear form) and the whole sequence $\{u_k\}\rightharpoonup u$ weakly in $H^1({\mathbb{R}^n})$.
	Now let's define the truncation $u_{k,\gamma}$ corresponding to $\gamma>0$ as follows,
	\begin{equation}
		u_{k,\gamma}(x) = 
		\left\{
		\begin{array}{ll}
			-\gamma & \text{if } u_k(x)\leq -\gamma,\\
			u_k(x) &  \text{if }-\gamma<u_k(x)<\gamma,\\
			\gamma & \text{if } u_k(x)\geq \gamma.
		\end{array}
		\right.
	\end{equation}
	Taking $v(x)=|u_{k,\gamma}(x)|^{q-2}u_{k,\gamma}(x)\in H^1(\mathbb{R}^n)$ as a test function in \eqref{seq_prob}, we get
	\begin{eqnarray*}
		&&(q-1)\int_{\mathbb{R}^n}|u_{k,\gamma}|^{q-2}\nabla u_k.\nabla u_{k,\gamma}dx+\beta\int_{\mathbb{R}^n}|u_{k,\gamma}|^{q-2}u_k(x).u_{k,\gamma}(x)dx \\
		&&+\frac{\lambda C(n,s)}{2}\int_{\mathbb{R}^n}\int_{\mathbb{R}^n}\frac{(u_k(x)-u_k(y))(|u_{k,\gamma}(x)|^{q-2}u_{k,\gamma}(x)-|u_{k,\gamma}(y)|^{q-2}u_{k,\gamma}(y))}{|x-y|^{n+2s}}dxdy\\
		&&=  \mu\int_{\mathbb{R}^n}(I_{\alpha}*M_ku_k)N_k|u_{k,\gamma}|^{q-2}u_{k,\gamma}dx+(\beta-1)\int_{\mathbb{R}^n}u(x)|u_{k,\gamma}(x)|^{q-2}u_{k,\gamma}(x)dx.
	\end{eqnarray*}
	By using lemma 3.5 of \cite{giacomoni2020regularity}, we can deduce that
	\begin{eqnarray*}
		\frac{\lambda C(n,s)}{2}\int_{\mathbb{R}^n}\int_{\mathbb{R}^n}\frac{(u_k(x)-u_k(y))(|u_{k,\gamma}(x)|^{q-2}u_{k,\gamma}(x)-|u_{k,\gamma}(y)|^{q-2}u_{k,\gamma}(y))}{|x-y|^{n+2s}}dxdy&&\\
		\geq \frac{4(q-1)\lambda C(n,s)}{2q^2}\int_{\mathbb{R}^n}\int_{\mathbb{R}^n}\frac{||u_{k,\gamma}(x)|^{\frac{q}{2}}-|u_{k,\gamma}(y)|^{\frac{q}{2}}|^2}{|x-y|^{n+2s}}dxdy,&&
	\end{eqnarray*}
	also, by construction of $u_{k,\gamma}$ we have $|u_{k,\gamma}|^q\leq |u_{k,\gamma}|^{q-2}u_{k,\gamma}u_k$, and hence 
	\begin{eqnarray*}
		&&     	(q-1)\int_{\mathbb{R}^n}|u_{k,\gamma}|^{q-2}|\nabla u_{k,\gamma}|^2 dx+\beta\int_{\mathbb{R}^n}|u_{k,\gamma}|^{q-2}u_{k,\gamma}u_kdx\\
		&& +\frac{\lambda C(n,s)}{2}\int_{\mathbb{R}^n}\int_{\mathbb{R}^n}\frac{(u_k(x)-u_k(y))(|u_{k,\gamma}(x)|^{q-2}u_{k,\gamma}(x)-|u_{k,\gamma}(y)|^{q-2}u_{k,\gamma}(y))}{|x-y|^{n+2s}}dxdy\\
		&& \geq \frac{4(q-1)}{q^2}\int_{\mathbb{R}^n}|\nabla (u_{k,\gamma})^{\frac{q}{2}}|^2dx+\beta\int_{\mathbb{R}^n}||u_{k,\gamma}|^{\frac{q}{2}}|^2dx+\frac{4(q-1)\lambda }{q^2}[|u_{k,\gamma}|^{\frac{q}{2}}]^2.
	\end{eqnarray*}
	Therefore,
	\begin{eqnarray*}
		&&\frac{4(q-1)}{q^2}\int_{\mathbb{R}^n}|\nabla |u_{k,\gamma}|^{\frac{q}{2}}|^2dx
		+\frac{4(q-1)\lambda }{q^2}[|u_{k,\gamma}|^{\frac{q}{2}}]^2
		+\beta\int_{\mathbb{R}^n}||u_{k,\gamma}|^{\frac{q}{2}}|^2dx\\
		& \leq &  (q-1)\int_{\mathbb{R}^n}|u_{k,\gamma}|^{q-2}\nabla u_{k,\gamma}.\nabla u_k dx+\beta\int_{\mathbb{R}^n}|u_{k,\gamma}|^{q-2}u_{k,\gamma}u_kdx\\
		&& +\frac{\lambda C(n,s)}{2}\int_{\mathbb{R}^n}\int_{\mathbb{R}^n}\frac{(u_k(x)-u_k(y))(|u_{k,\gamma}(x)|^{q-2}u_{k,\gamma}(x)-|u_{k,\gamma}(y)|^{q-2}u_{k,\gamma}(y))}{|x-y|^{n+2s}}dxdy\\
		& = &\mu\int_{\mathbb{R}^n}(I_{\alpha}*M_ku_k)N_k|u_{k,\gamma}|^{q-2}u_{k,\gamma}dx+(\beta-1)\int_{\mathbb{R}^n}|u_{k,\gamma}|^{q-2}u_{k,\gamma}udx.
	\end{eqnarray*}
	For $q<\frac{2n}{\alpha}$, taking $\theta=\frac{2}{q}$ and $u=|u_{k,\gamma}|^{\frac{q}{2}}$ in Lemma 3.2 of \cite{moroz2015existence}, we get
	\begin{eqnarray*}
		&&\mu\int_{\mathbb{R}^n}(I_{\alpha}*M_ku_{k,\gamma})N_k|u_{k,\gamma}|^{q-2}u_{k,\gamma}\\ 
		& &\leq  \mu\int_{\mathbb{R}^n}(I_{\alpha}*M_k|u_{k,\gamma}|)N_k|u_{k,\gamma}|^{q-2}|u_{k,\gamma}|\\
		&&\leq  \frac{2(q-1)}{q^2}\int_{\mathbb{R}^n}|\nabla(u_{k,\gamma})^{\frac{q}{2}}|^2dx+C''\int_{\mathbb{R}^n}||u_{k,\gamma}|^{\frac{q}{2}}|^2dx  \\
		&&\leq  \frac{2(q-1)}{q^2}\int_{\mathbb{R}^n}|\nabla(u_{k,\gamma})^{\frac{q}{2}}|^2dx+C''\int_{\mathbb{R}^n}||u_{k,\gamma}|^{\frac{q}{2}}|^2dx
		+\frac{2(q-1)}{q^2}\lambda [|u_{k,\gamma}|^{\frac{q}{2}}|]^2.
	\end{eqnarray*}
	Now, taking $A_{k,\gamma}=\{x\in \mathbb{R}^n: |u_k(x)|>\gamma\}$, we get
	\begin{eqnarray*}
		&&   	\frac{4(q-1)}{q^2}\left(\int_{\mathbb{R}^n}|\nabla |u_{k,\gamma}|^{\frac{q}{2}}|^2dx
		+\lambda [|u_{k,\gamma}|^{\frac{q}{2}}]^2\right)
		+\beta\int_{\mathbb{R}^n}||u_{k,\gamma}|^{\frac{q}{2}}|^{2}dx	-C''\int_{\mathbb{R}^n}||u_{k,\gamma}|^{\frac{q}{2}}|^2dx\\
		&&-\frac{2(q-1)}{q^2}\left(\int_{\mathbb{R}^n}|\nabla|u_{k,\gamma}|^{\frac{q}{2}}|^2dx+\lambda[|u_{k,\gamma}|^{\frac{q}{2}}]^2\right)\\
		&&	\leq 2\mu\int_{A_{k,\gamma}}(I_{\alpha}*|N_k||u_k|^{q-1})|M_ku_k|+(\beta-1)\int_{\mathbb{R}^n}|u_{k,\gamma}|^{q-2}u_{k,\gamma}udx,
	\end{eqnarray*}
	that is,
	\begin{eqnarray}\label{4.4}
		&&\frac{2(q-1)}{q^2}\left(\int_{\mathbb{R}^n}|\nabla|u_{k,\gamma}|^{\frac{q}{2}}|^2dx+\lambda[|u_{k,\gamma}|^{\frac{q}{2}}]^2\right)\nonumber\\
		&&\leq     	2\mu\int_{A_{k,\gamma}}(I_{\alpha}*|N_k||u_k|^{q-1})|M_ku_k|+C_1\left[\int_{\mathbb{R}^n}|u_{k,\gamma}|^{q-2}u_{k,\gamma}u+\int_{\mathbb{R}^n}|u_{k,\gamma}|^qdx\right]\nonumber\\
		&&\leq C_1\left[\int_{\mathbb{R}^n}|u_k|^q+|u|^q\right]+2\mu\int_{A_{k,\gamma}}(I_{\alpha}*|N_k||u_k|^{q-1})|M_ku_k|.
	\end{eqnarray}
	Now, by Hardy-Littlewood-Sobolev inequality,
	\begin{eqnarray*}
		\mu\int_{A_{k,\gamma}}(I_{\alpha}*|N_k||u_k|^{q-1})|M_ku_k|& \leq & C\left(\int_{\mathbb{R}^n}||N_k||u_k|^{q-1}|^r\right)^{\frac{1}{r}}\left(\int_{A_{k,\gamma}}|M_ku_k|^s\right)^{\frac{1}{s}},
	\end{eqnarray*}
	here $r=\frac{2nq}{\alpha q+2nq-2n}$ and $t =\frac{2nq}{\alpha q+2n}$. If $u_k\in L^q(\mathbb{R}^n)$, then using H$\ddot{\text{o}}$lder's inequality we can see that $|N_k||u_k|^{q-1}\in L^{r}(\mathbb{R}^n)$ and $|M_ku_k|\in L^t(\mathbb{R}^n)$ and hence 
	\begin{equation}\label{4.5}
		\mu\int_{A_{k,\gamma}}(I_{\alpha}*|N_k||u_k|^{q-1})|M_ku_k|\rightarrow 0 \text{ as } \gamma\rightarrow \infty.
	\end{equation}
	Taking $\gamma\rightarrow \infty$ in \eqref{4.4} and using \eqref{4.5}, we get
	\begin{eqnarray*}
		\frac{2(q-1)}{q^2}\left(\int_{\mathbb{R}^n}|\nabla|u_k|^{\frac{q}{2}}|^2+\lambda[|u_k|^{\frac{q}{2}}]^2\right)
		&\leq & C_1\left(\int_{\mathbb{R}^n}|u_k|^q+|u|^q\right).
	\end{eqnarray*}
	This implies that $|u_k|^{\frac{q}{2}}\in H^1(\mathbb{R}^n)\hookrightarrow L^{\frac{2n}{n-2}}(\mathbb{R}^n)$ and hence by Fatou's lemma
	\begin{eqnarray*}
		\left(\int_{\mathbb{R}^n}||u_k|^{\frac{q}{2}}|^{\frac{2n}{n-2}}\right)^{\frac{2(n-2)}{2n}}& \leq & K_1\left\||u_k|^{\frac{q}{2}} \right\|^2 \leq K'_1\left(\int_{\mathbb{R}^n}|u_k|^q+|u|^q\right)\\
		& \leq &K'_1\left(\int_{\mathbb{R}^n}|u_k|^q+\limsup_{k\rightarrow \infty}\int_{\mathbb{R}^n}|u_k|^q\right).
	\end{eqnarray*}
	Thus, again by Fatou's lemma
	\begin{equation}\label{higherIntegrability}
		\left(\int_{\mathbb{R}^n}|u|^{\frac{nq}{n-2}}\right)^{\frac{n-2}{n}}\leq \limsup_{k\rightarrow \infty}\left(\int_{\mathbb{R}^n}|u_k|^{\frac{nq}{n-2}}\right)^{1-\frac{2}{n}}\leq K''\limsup_{k\rightarrow \infty}\int_{\mathbb{R}^n}|u_k|^q.
	\end{equation}
	Starting with $q=2$, using \eqref{higherIntegrability} we get the required result as follows:\\
	Case 1: $n-\alpha\leq2$.\\
	For $n-\alpha\leq 2$, we have $\frac{2n}{\alpha}\leq 2^*$ and hence since $u,u_k\in L^m(\mathbb{R}^n)$ for all $m\in [2,\frac{2n}{\alpha})$ and $k\in \mathbb{N}$, by \eqref{higherIntegrability} we get $u\in L^m(\mathbb{R}^n)$ for all $m\in [2,\frac{n}{\alpha}2^*)$.\\
	Case 2: $n-\alpha>2$\\
	In this case, we have $\frac{2n}{\alpha}>2^*$. Now taking $0<\epsilon\leq 2^*-2$, we get $u$, $u_k\in L^m(\mathbb{R}^n)$ for all $m\in [2,2+\epsilon)\subset [2,2^*]$ and $k\in \mathbb{N}$, hence by \eqref{higherIntegrability}   
	\begin{equation}\label{case2}
		u, u_k\in L^m(\mathbb{R}^n) \text{ for all } m\in \left[2,(2+\epsilon)\frac{n}{n-2}\right) \text{ and } k\in \mathbb{N}.
	\end{equation}
	If $(2+\epsilon)\frac{n}{n-2}\geq \frac{2n}{\alpha}$ then as done in case 1, we get the required result, thus we stop the iteration, and for the subcase $(2+\epsilon)\frac{n}{n-2} < \frac{2n}{\alpha}$, using \eqref{higherIntegrability} and \eqref{case2} we will get 
	\begin{equation*}
		u, u_k\in L^m(\mathbb{R}^n) \text{ for all } m\in \left[2,(2+\epsilon)\left(\frac{n}{n-2}\right)^2\right) \text{ and } k\in \mathbb{N}.
	\end{equation*}
	Again if $(2+\epsilon)\left(\frac{n}{n-2}\right)^2\geq \frac{2n}{\alpha}$, the iteration terminates, otherwise we use the above higher integrability of $u$, $u_k$ and \eqref{higherIntegrability}  to increase the range of $m$. Now since the map $t\mapsto (2+\epsilon)\left(\frac{n}{n-2}\right)^t$ increases to infinity, 
	the procedure must stop after a finite number of iterations. Hence $u\in L^m(\mathbb{R}^n)$ for all $m\in [2,\frac{n}{\alpha}2^*)$.
\end{proof}
\begin{myproof}{Theorem}{\ref{thm 1.1}}
	
	Let $u$ be a solution of \eqref{prob}. Since $\frac{n+\alpha}{n}\leq p\leq \frac{n+\alpha}{n-2}$, we get for any $x\in\mathbb{R}^n$
	$$|u(x)|^{p-1}\leq \left(|u(x)|^{\frac{\alpha}{n}}+|u(x)|^{\frac{\alpha+2}{n-2}}\right).$$
	Defining $N,M$ by $N(x)=M(x)=|u(x)|^{p-2}u(x)$ for any $x\in\mathbb{R}^n$, clearly $N,M\in L^{\frac{2n}{\alpha}}(\mathbb{R}^n)+L^{\frac{2n}{\alpha+2}}(\mathbb{R}^n)$. Therefore, by proposition 2.1, $u\in L^q(\mathbb{R}^n)$ for $q\in [2,\frac{n}{\alpha}\frac{2n}{n-2})$.
\end{myproof}
\begin{myproof}{Theorem}{\ref{thm1.2}}
	We know that, for all $x\in \mathbb{R}^n$ 
	\begin{equation*}
		(I_{\alpha}*|u|^p)(x) = \int_{\mathbb{R}^N}\frac{A_{\alpha}|u(x-y)|^p}{|y|^{n-\alpha}}dy = \int_{B(0,1)}\frac{A_{\alpha}|u(x-y)|^p}{|y|^{n-\alpha}}dy+\int_{\mathbb{R}^n\setminus B(0,1)}\frac{A_{\alpha}|u(x-y)|^p}{|y|^{n-\alpha}}dy.
	\end{equation*}
	For some $1<\frac{n}{\alpha}<\gamma<\frac{2n}{\alpha p}(\frac{n}{n-2})$, using \autoref{thm 1.1} and the H$\ddot{\text{o}}$lder's inequality, we get
	\begin{equation*}
		\int_{B(0,1)}\frac{A_{\alpha}|u(x-y)|^p}{|y|^{n-\alpha}}dy  \leq  C_\alpha \left(\int_{B(0,1)}\frac{dy}{|y|^{\frac{(n-\alpha)\gamma}{\gamma-1}}}\right)^{\frac{\gamma-1}{\gamma}}<M_1,
	\end{equation*}
	also, we have:
	\begin{eqnarray*}
		\int_{\mathbb{R}^n\setminus B(0,1)}\frac{A_{\alpha}|u(x-y)|^p}{|y|^{n-\alpha}}dy  \leq  A_{\alpha} \left(\int_{\mathbb{R}^n\setminus B(0,1)}|u(x-y)|^{\frac{2np}{n+\alpha}}\right)^{\frac{n+\alpha}{2n}}\left(\int_{\mathbb{R}^n\setminus B(0,1)}\frac{dy}{|y|^{(n-\alpha)\frac{2n}{n-\alpha}}}\right)^{\frac{n-\alpha}{2n}}.
	\end{eqnarray*}
	Therefore, there exists $M>0$ such that
	$(I_{\alpha}*|u|^p)(x)<M$,
	for all $\frac{n+\alpha}{n}\leq p \leq \frac{n+\alpha}{n-2}$. Thus,
	\begin{equation}\label{I_alpha_L_infinity}
		(I_{\alpha}*|u|^p)\in L^{\infty}(\mathbb{R}^n) \text{ for all } \frac{n+\alpha}{n}\leq p \leq \frac{n+\alpha}{n-2}.
	\end{equation}	
	For $0<\epsilon <1$, defining $h_{\epsilon}\,:\, t\to h_{\epsilon}(t):=\sqrt{\epsilon^2+t^2}$, clearly, $h_{\epsilon}$ is a convex function. Now, taking $g_{\epsilon}:=h_{\epsilon}'$, by \cite[Theorem 2.2.3]{Kesavan2019Topics} we get $g_{\epsilon}(u)\in H^1(\mathbb{R}^n)$, since $g_{\epsilon}$ is continuously differentiable, $g_{\epsilon}(0)=0$ and $|g'_{\epsilon}(t)|\leq \frac{1}{\epsilon}:=M_{\epsilon}$.
	Consider, $\zeta= \phi g_{\epsilon}(u)\in H^1(\mathbb{R}^n)$ for some $0\leq\phi \in C_c^{\infty}(\mathbb{R}^n)$ chosen arbitrarily, then from convexity of $h_{\epsilon}$ we have $\phi(x)h'_{\epsilon}(u(x))(u(x)-u(y))\geq \phi(x)(h_{\epsilon}(u(x))-h_{\epsilon}(u(y)))$ (see \cite[Lemma~A.1]{Brasco2016second}) and then we deduce easily by interchanging the role of $x$ and $y$:
	\begin{equation}\label{eq1.2.1}
		\ll u, \zeta\gg \geq \frac{C(n,s)}{2}\int_{\mathbb{R}^n}\int_{\mathbb{R}^n}\frac{(h_{\epsilon}(u(x))-h_{\epsilon}(u(y)))(\phi(x)-\phi(y))}{|x-y|^{n+2s}}dxdy = \ll h_{\epsilon}(u), \phi\gg.
	\end{equation}
	Thus, taking $\zeta$ as a test function in \eqref{prob}, we have: 
	\begin{eqnarray*}
		&& \int_{\mathbb{R}^n}|\nabla u|^2\phi g_{\epsilon}'(u)+\int_{\mathbb{R}^n}g_{\epsilon}(u)\nabla u \nabla \phi +\lambda \ll h_{\epsilon}(u), \phi \gg +\int_{\mathbb{R}^n}u\zeta \\
		&& \leq  \int_{\mathbb{R}^n}\nabla u \nabla \zeta +\lambda\ll u,\zeta \gg +\int_{\mathbb{R}^n}u\zeta = \mu \int_{\mathbb{R}^n}(I_{\alpha}*|u|^p)|u|^{p-2}u\zeta\\
		&& \leq  \mu \int_{\mathbb{R}^n}(I_{\alpha}*|u|^p)|u|^{p-1}\phi.
	\end{eqnarray*}
	Taking $\epsilon \rightarrow 0$ in the above expression together with dominated convergence theorem, we get for any nonnegative $\phi\in C^\infty_c(\mathbb{R}^n)$:
	\begin{equation}\label{eq1.2.2}
		\int_{\mathbb{R}^n}\nabla |u| \nabla \phi+\lambda \ll |u|, \phi\gg+\int_{\mathbb{R}^n}|u|\phi \leq \mu \int_{\mathbb{R}^n}(I_{\alpha}*|u|^p)|u|^{p-1}\phi.
	\end{equation}
	By density, the above inequality holds for all nonnegative $\phi \in H^1(\mathbb{R}^n)$. Now, for a fixed $\gamma>0$, define $u_{\gamma}:=\min\{\gamma, |u|\}$ and, for some $k>1$, set $\beta =2k-1$. Replacing $\phi$ by $u_{\gamma}^{\beta}\in H^1(\mathbb{R}^n)$ in \eqref{eq1.2.2} and using \cite[Lemma 3.1]{Biswas2023Regularity}, we get
	\begin{eqnarray*}
		&&\frac{4\beta}{(\beta+1)^2}\left(\left\| \nabla u_{\gamma}^k\right\|_2^2+\lambda [u_{\gamma}^k]^2\right)+\left\| u_{\gamma}^k\right\|_2^2 \\
		&& \leq  {\beta}\int_{\{|u|<\gamma\}}|\nabla |u||^2 u_{\gamma}^{\beta-1}+\lambda\ll |u|,u_{\gamma}^{\beta}\gg +\int_{\mathbb{R}^n}|u|u_{\gamma}^{\beta}\\
		&& = \int_{\mathbb{R}^n}\nabla |u| \nabla u_{\gamma}^{\beta}+\lambda\ll |u|, u_{\gamma}^{\beta}\gg +\int_{\mathbb{R}^n}|u| u_{\gamma}^{\beta} \leq \mu \int_{\mathbb{R}^n}(I_{\alpha}*|u|^p)|u|^{p-1}u_{\gamma}^{\beta}.
	\end{eqnarray*}
	Therefore,
	$$\left\| u_{\gamma}^k\right\|_{H^1(\mathbb{R}^n)}^2\leq \frac{k^2}{\beta}\mu \int_{\mathbb{R}^n}(I_{\alpha}*|u|^p)|u|^{p-1}u_{\gamma}^{\beta}.$$
	Now, by the continuous imbedding $H^1(\mathbb{R}^n)\hookrightarrow L^{2^*}(\mathbb{R}^n)$, \autoref{prop1.1}, H$\ddot{\text{o}}$lder's inequality, the fact that $k>1$ and $(a+b)^c<a^c+b^c$ for all $a,b>0$, $c<1$, for some $\delta>1$, we have:
	\begin{eqnarray*}
		\left\| u_{\gamma}^k \right\|_{2^*}^2 & \leq & C_1 \left\| u_{\gamma}^{k} \right\|_{H^1(\mathbb{R}^n)}^2 \leq  C_1k\mu\int_{\mathbb{R}^n}(I_{\alpha}*|u|^p)|u|^{p-1}u_{\gamma}^{\beta}\\
		& \leq & C_1k\mu C_n\left(\int_{\mathbb{R}^n}|u|^{\frac{2np}{n+\alpha}}\right)^{\frac{n+\alpha}{2n}}\left(\int_{\mathbb{R}^n}(|u|^{p-2}|uu_{\gamma}^{\beta}|)^{\frac{2n}{n+\alpha}}\right)^{\frac{n+\alpha}{2n}}\\
		& \leq & C_1k\mu \bar{C_n}\left(\int_{\{|u|<\delta\}}(|u|^{p-2}|uu_{\gamma}^{\beta}|)^{\frac{2n}{n+\alpha}}+\int_{\{|u|\geq\delta\}}(|u|^{p-2}|uu_{\gamma}^{\beta}|)^{\frac{2n}{n+\alpha}}\right)^{\frac{n+\alpha}{2n}}\\
		& \leq & C_1k\mu \bar{C_n}\left(\left(\int_{\{|u|<\delta\}}(|u|^{p-2}|uu_{\gamma}^{\beta}|)^{\frac{2n}{n+\alpha}}\right)^{\frac{n+\alpha}{2n}}+\left(\int_{\{|u|\geq\delta\}}(|u|^{p-2}|uu_{\gamma}^{\beta}|)^{\frac{2n}{n+\alpha}}\right)^{\frac{n+\alpha}{2n}}\right)\\
		& \leq & C_1k\mu \bar{C_n}\delta^{p-2}\left\| u \right\|_{L^{\frac{4nk}{n+\alpha}}(\mathbb{R}^n)}^{2k}+C_1k\mu\bar{C_n}\left(\int_{\{|u|\geq\delta\}}|u|^{(\frac{\alpha+4-n}{n-2})(\frac{2n}{n+\alpha})}|u|^{\frac{4nk}{n+\alpha}}\right)^{\frac{n+\alpha}{2n}}\\
		& \leq & C_1k\mu \bar{C_n}\delta^{p-2}\left\| u \right\|_{L^{\frac{4nk}{n+\alpha}}(\mathbb{R}^n)}^{2k}+C_1k\mu\bar{C_n}\left(\int_{\{|u|\geq\delta\}}|u|^{2^*}\right)^{\frac{\alpha+4-n}{2n}}\left(\int_{\{|u|\geq\delta\}}|u|^{2^*k}\right)^{\frac{n-2}{n}}\\
		& \leq & C_1k\mu \bar{C_n}\delta^{p-2}\left\| u \right\|_{L^{\frac{4nk}{n+\alpha}}(\mathbb{R}^n)}^{2k}+C_1k\mu\bar{C_n}C(\delta)^{\frac{\alpha+4-n}{n-2}}\left\| u \right\|_{L^{2^*k}(\mathbb{R}^n)}^{2k}.
	\end{eqnarray*}
	Taking $\delta=\delta_k>1$ large enough such that $C_2:=C_1k\mu\bar{C_n}C(\delta_k)^{\frac{\alpha+4-n}{n-2}}<1$ together with Fatou's lemma, we obtain:
	$$\left\| u \right\|_{L^{{2^*k}}(\mathbb{R}^n)}^{2k}=\left\| u^k \right\|_{2^*}^2\leq \liminf_{\gamma\rightarrow \infty} \left\| u_{\gamma}^k\right\|_{2^*}^2\leq C_1k\mu \bar{C_n}\delta_k^{p-2}\left\| u \right\|_{L^{\frac{4nk}{n+\alpha}}(\mathbb{R}^n)}^{2k}+C_2\left\| u \right\|_{L^{2^*k}(\mathbb{R}^n)}^{2k}$$
	and
	\begin{equation}\label{eq1.2.3}
		\left\| u \right\|_{2^*k} \leq (\hat{C_k})^{\frac{1}{2k}}(k)^{\frac{1}{2k}}\left\| u \right\|_{\frac{4nk}{n+\alpha}},
	\end{equation}
	where $\hat{C_k}=\frac{\mu C_1\bar{C_n}\delta_k^{p-2}}{1-C_2}$. Now, since $u\in L^{\frac{4nk}{n+\alpha}}(\mathbb{R}^n)$ for all $k\in \left(1, (\frac{n+\alpha}{4\alpha})2^*\right)$, then by \eqref{eq1.2.3}, $$u\in L^q(\mathbb{R}^n) \text{ for all } q\in \left(2^*, \gamma_1\right),$$
	where $\gamma_1=\left(\frac{n+\alpha}{4\alpha}\right)(2^*)^2=\frac{n2^*}{\alpha}\left(\frac{2^*(n+\alpha)}{4n}\right)$ and since $2\leq 2^*<\frac{n}{\alpha}2^*$, we get $u\in L^q(\mathbb{R}^n)$ for all $q\in [2,\gamma_1)$. Now, again using \eqref{eq1.2.3} in a similar way, we get $u\in L^{q}(\mathbb{R}^n)$ for all $q\in [2,\gamma_2)$ where $\gamma_2=\frac{n2^*}{\alpha}\left(\frac{2^*(n+\alpha)}{4n}\right)^2$ and so on by iteration. Now, since $\frac{2^*(n+\alpha)}{4n}>1$ for $n-\alpha<4$, we have  $\gamma_r=\frac{n2^*}{\alpha}\left(\frac{2^*(n+\alpha)}{4n}\right)^r\rightarrow\infty$ as $r\rightarrow \infty$. Therefore, 
	\begin{equation}\label{eq1.2.4}
		u\in L^q(\mathbb{R}^n) \text{ for all } q\in [2,\infty).
	\end{equation}
	Defining $g:=\mu(I_{\alpha}*|u|^{p})|u|^{p-2}u-u$, one has $g\in L^q(\mathbb{R}^n)$ for all $q\in [2,\infty)\subset [\frac{2}{p-1},\infty)$ and 
	\begin{eqnarray*}
		Tail_{1,s,2}(u,x_0,R) & = & R^2\int_{\mathbb{R}^n\setminus B_R(x_0)}\frac{|u|}{|x-x_0|^{N+s}}dx\\
		& \leq & R^2\left(\int_{\mathbb{R}^n\setminus B_R(x_0)}|u|^2dx\right)^{\frac{1}{2}}\left(\int_{\mathbb{R}^n\setminus B_R(x_0)}\frac{1}{|x-x_0|^{2(N+s)}}dx\right)^{\frac{1}{2}}<+\infty.
	\end{eqnarray*}
	Thus, by \cite[Theorem 1.2]{Garain2023Higher}, $u\in L^{\infty}_{loc}(\mathbb{R}^n)$ (by \cite[Theorem, 1.4]{Garain2023Higher}, it is even $u\in C^{\delta}_{loc}(\mathbb{R}^n)$ for every $0<\delta<\min\{2s,1\}$).
	Now, we reformulate our problem: For fixed $B>0$, let us define the following operators:
	$$\mathcal{L}_1(u):=C(n,s)\text{P.V.}\int_{|x-y|\leq B}\frac{u(x)-u(y)}{|x-y|^{n+2s}}dy \text{ and }\mathcal{L}_2(u):=C(n,s)\text{P.V.}\int_{|x-y|> B}\frac{u(x)-u(y)}{|x-y|^{n+2s}}dy.$$
	Thus we can write \eqref{prob} as
	$$-\Delta u+\lambda \mathcal{L}_1(u)+u=\mu (I_{\alpha}*|u|^p)|u|^{p-2}u-\mathcal{L}_2(u) \text{ in } \mathbb{R}^n.$$
	Defining $f:=\mu(I_{\alpha}*|u|^p)|u|^{p-2}u-\mathcal{L}_2(u)$, using $u\in L^{\infty}_{loc}(\mathbb{R}^n)$ and $L^q(\mathbb{R}^n)$ for all $q\in [2,\infty)$ and applying Fubini's theorem together with:
	\begin{eqnarray*}
		&\left(\int_{\mathbb{R}^n}\left|\int_{|x-y|> B}\frac{u(x)-u(y)}{|x-y|^{n+2s}}dy\right|^qdx\right)dx\leq C(q)\int_{\mathbb{R}^n}|u|^q\left(\int_{|x-y|> B}\frac{1}{|x-y|^{(n+2s)q}}dy\right)dx<\\
		&2C(B,q)\int_{\mathbb{R}^n}|u|^q<+\infty
	\end{eqnarray*}
	and 
	$$\int_{\mathbb{R}^n}|(I_{\alpha}*|u|^p)|u|^{p-1}|^q\leq M^q\int_{\mathbb{R}^n}|u|^{q(p-1)}<+\infty,$$
	we get $f\in L^q(\mathbb{R}^n)$ for all $q\geq 2$. Hence, using \cite[Theorem 3.1.20]{garroni2002second} with $\Omega_{I}$ being the ball of radius $B$ centered at $0$ and a fixed bounded domain $\Omega$, we get $u\in W^{2,q}_{loc}(\mathbb{R}^n)$ for all $q\geq 2$.
\end{myproof}
\section{Existence and Equivalence Results}
\noindent	This section is devoted to the existence and equivalence results for the solution of \eqref{prob}. In that context, we start with the following observations:\\
Let $u\in  H^1(\mathbb{R}^n)$ be a solution of 
\begin{equation}\label{1.1}
\mathcal{L}u+u=\mu (I_\alpha*|u|^p)|u|^{p-2}u\;\;\;\text{in } \mathbb{R}^n.
\end{equation}
Then, u is a critical point of $S_\lambda$,
this gives 
\begin{equation}\label{NehariI}
\left\| \nabla u \right\|_2^2+ \left\| u \right\|_2^2+\lambda[u]^2=\mu \mathcal{A}_p(u),
\end{equation}
and 
by \autoref{Pohozaev_identity} the following Pohozaev type identity holds:
\begin{eqnarray}\label{PohozaevI}
\frac{\mu (n+\alpha)}{2p}\mathcal{A}_p(u) & = & \frac{ (n-2s)}{2}\lambda [u]^2
+\frac{n}{2}\left\| u \right\|_2^2
+\left(\frac{n-2}{2}\right)\left\| \nabla u \right\|_2^2.
\end{eqnarray}
From \eqref{NehariI} and \eqref{PohozaevI}, we deduce that
\begin{equation*}
\frac{\mu}{2p}\mathcal{A}_p(u)
=  \frac{\left\| u \right\|_2^2}{n+\alpha-p(n-2)}
+\frac{ (1-s)}{(n+\alpha-p(n-2))}\lambda[u]^2
\end{equation*}
and thus
\begin{equation*}
S_{\lambda}(u)  =  \frac{\mu(p-1)}{2p}\mathcal{A}_p(u)
=  \frac{(p-1)\left\| u \right\|_2^2}{n+\alpha-p(n-2)}
+\frac{\lambda (1-s)(p-1)[u]^2}{(n+\alpha-p(n-2))} >0
\end{equation*}
for $\frac{n+\alpha}{n}<p<\frac{2s+n+\alpha}{n}\text { and } s\in(0,1).$	Now, if $u\in G$, we get
$\min\{ S_{\lambda}(u):\;u\in A\}=l>0.$  
Now, let us define
$$I_{\lambda}(u):=\frac{\left\| \nabla u \right\|_2^2}{2}+\frac{\lambda [u]^2}{2}- \frac{\mu \mathcal{A}_p(u)}{2p}
.$$
The following Gagliardo-Nirenberg inequality for Choquard nonlinearity will help us to study the boundedness of the functional $I$ on $\Gamma$.
\begin{pro}
For $\frac{n+\alpha}{n}\leq p\leq \frac{2s+n+\alpha}{n}$, 
there exists a positive constant $C_{n,p}$ such that :
\begin{equation}\label{G-N}
	\mathcal{A}_p(u)
	\leq C_{n,p}\left(\int_{\mathbb{R}^n}|\nabla u|^2dx\right)^{\frac{np-n-\alpha}{2}}\left(\int_{\mathbb{R}^n}|u|^2dx\right)^{\frac{n+\alpha-p(n-2)}{2}}
\end{equation}
for every $u\in H^1(\mathbb{R}^n)$.
\end{pro}
\begin{proof}
Let $r, t>1$ be such that $\frac{1}{r}+\frac{1}{t}=1+\frac{\alpha}{n}$, then by the Hardy-Littlewood-Sobolev inequality we have:
$$\mathcal{A}_p(u)
=\int_{\mathbb{R}^n}\int_{\mathbb{R}^n}\frac{A_{\alpha}|u(x)|^p|u(y)|^p}{|x-y|^{n-\alpha}}dxdy\leq C_{n,r,s}\left(\int_{\mathbb{R}^n}|u|^{pr}dx\right)^{\frac{1}{r}}\left(\int_{\mathbb{R}^n}|u|^{pt}dx\right)^{\frac{1}{t}},$$
for $r,t\in (1,\frac{n}{\alpha})\cap [\frac{2}{p},\frac{2n}{p(n-2)}]$. Since $1<\frac{2n}{p(n-2)}$ and $\frac{2}{p}<\frac{n}{\alpha}$, we can find such $r$ and $t$.
Now taking $q>0$ such that $q+1=pr$, then by Gagliardo-Nirenberg inequality (see \cite[Theorem 1.1]{Fiorenza2021detailed}), precisely, $\left\| u \right\|_\beta\leq C_{n,\beta}\left\| \nabla u \right\|_2^{\theta}\left\| u \right\|_2^{1-\theta}$ where $\theta=\frac{N(\beta-2)}{2\beta}$, we have:
\begin{eqnarray*}
	\left(\int_{\mathbb{R}^n}|u|^{pr}\right)^\frac{1}{r} &
	\leq & \left(C_{n,q}\left(\int_{\mathbb{R}^n}|\nabla u|^2dx\right)^{\frac{n(q-1)}{4}}\left(\int_{\mathbb{R}^n}|u|^2dx\right)^{\frac{(q+1)}{2}-\frac{n(q-1)}{4}}\right)^{\frac{p}{q+1}}\\
	& = & C_{n,r}\left(\int_{\mathbb{R}^n}|\nabla u|^2dx\right)^{\frac{n(pr-2)}{4r}}\left(\int_{\mathbb{R}^n}|u|^2dx\right)^{\frac{p}{2}-\frac{n(pr-2)}{4r}}.
\end{eqnarray*}
Similarly, we can deduce that
$$\left(\int_{\mathbb{R}^n}|u|^{p{t}}\right)^\frac{1}{t} \leq C_{n,t}\left(\int_{\mathbb{R}^n}|\nabla u|^2dx\right)^{\frac{n(p{t}-2)}{4s}}\left(\int_{\mathbb{R}^n}|u|^2dx\right)^{\frac{p}{2}-\frac{n(p{t}-2)}{4s}}.$$
Thus we get
\begin{equation}\label{1}
	\mathcal{A}_p(u)
	\leq C_{n,r,{t}}\left(\int_{\mathbb{R}^n}|\nabla u|^2\right)^{\frac{n(pr-2)}{4r}+\frac{n(p{t}-2)}{4{t}}}\left(\int_{\mathbb{R}^n}|u|^2dx\right)^{p-\frac{n(pr-2)}{4r}-\frac{n(p{t}-2)}{4{t}}}.
\end{equation}
Since $\frac{1}{r}+\frac{1}{{t}}=1+\frac{\alpha}{n}$, we get ${t}=\frac{rn}{rn+\alpha r-n}$ which implies
$\frac{n(p{t}-2)}{4{t}}=\frac{prn-2rn-2\alpha r +2n}{4r}$
and hence
\begin{equation}\label{2}
	\frac{n(pr-2)}{4r}+\frac{n(p{t}-2)}{4{t}}=\frac{n(pr-2)}{4r}+\frac{prn-2rn-2\alpha r +2n}{4r}=\frac{np-n-\alpha}{2}.
\end{equation}
Note that taking such $r$ and $t$ makes sense because for $\frac{n+\alpha}{n}\leq p \leq \frac{2s+n+\alpha}{n}$, the following inequalities hold:
$$\frac{2n}{2(n+\alpha)-p(n-2)}\leq \frac{2n}{p(n-2)},\;\;\frac{2}{p}\leq \frac{2n}{2(n+\alpha)-np},$$
$$\frac{2n}{2(n+\alpha)-p(n-2)}<\frac{n}{\alpha}\text{ and } 1<\frac{2n}{2(n+\alpha)-np},$$
hence $(1,\frac{n}{\alpha})\cap[\frac{2}{p},\frac{2n}{p(n-2)}]\cap [\frac{2n}{2(n+\alpha)-p(n-2)},\frac{2n}{2(n+\alpha)-np}]$ is non-empty. Thus,
using \eqref{2} in \eqref{1}, we get:
$$\mathcal{A}_p(u)
\leq C_{n,p}\left(\int_{\mathbb{R}^n}|\nabla u|^2\right)^{\frac{np-n-\alpha}{2}}\left(\int_{\mathbb{R}^n}|u|^2dx\right)^{\frac{n+\alpha-p(n-2)}{2}}.$$
\end{proof}
\begin{lemma}\label{lemma3.1}
Define $-m_{\lambda} :=\inf\{I_{\lambda}(v):v\in \Gamma\}$, then for $\alpha \in (0,n)$ and $s\in(0,1)$, we get
\begin{enumerate}
	\item $m_{\lambda}>0$ whenever $\frac{n+\alpha}{n}<p<\frac{2s+n+\alpha}{n}$,
	\item $m_{\lambda}\leq 0$ for $p=\frac{2+n+\alpha}{n}$ and suitably small $\mu$,
	\item the functional $I$ is unbounded below for $\frac{2+n+\alpha}{n}<p$.
\end{enumerate}
\end{lemma}	
\begin{proof}
Claim : I is bounded below on $\Gamma$, for  $\frac{n+\alpha}{n}< p\leq\frac{2+n+\alpha}{n}$.\\
Proof of Claim. Let $u\in \Gamma$ be arbitrary, using \eqref{G-N} and Young's inequality, we deduce that
\begin{eqnarray*}
	I_{\lambda}(u) & \geq & \frac{1}{2}\left\| \nabla u\right\|_2^2- \frac{\mu C_{n,p}}{2p}(\left\| u\right\|_2^2)^{\frac{n+\alpha-p(n-2)}{2}}(\left\| \nabla u\right\|_2^2)^{\frac{np-n-\alpha}{2}}\\
	& \geq & \frac{1}{2}\left\| \nabla u\right\|_2^2-\frac{a^\kappa}{\kappa}-\frac{\left\| \nabla u\right\|_2^2(np-n-\alpha)}{2.2^\beta}
\end{eqnarray*}
with $a=\frac{C_{n,p}\mu}{p}(\tau)^{\frac{n+\alpha-p(n-2)}{2}}$, $\beta=\frac{2}{np-n-\alpha}>1$ and $\kappa=\frac{2}{n+\alpha+2-np}>1$. Therefore,
\begin{equation}\label{2.3}
	I_{\lambda}(u)\geq \frac{1}{2}\left(1-\frac{np-n-\alpha}{2}\right)\left\| \nabla u \right\|_2^2-\frac{a^\kappa}{\kappa} \geq -\frac{a^\kappa}{\kappa},\;\;\forall\,u\in \Gamma.
\end{equation}
\begin{itemize}
	\item[Proof of 1.] Let $u\in \Gamma$ and  $K>0$ be arbitrary. We define
	$u_K(x)=K^{\frac{n}{2}}u(Kx).$
	Clearly, $u_K\in \Gamma$ and 
	\begin{eqnarray}\label{Kargument}
		I_{\lambda}(u_K) 	&=&\frac{K^2}{2}\left\| \nabla u \right\|_2^2+\frac{K^{2s}}{2}\lambda[u]^2
		-\frac{\mu }{2p}K^{np-n-\alpha}\mathcal{A}_p(u).
	\end{eqnarray}
	Since $p<\frac{2s+n+\alpha}{n}$, we deduce that $I_{\lambda}(u_K)<0$ for small $K>0$. Hence $-m_{\lambda}<0.$
	\item[Proof of 2.]	Let $u\in \Gamma$ be arbitrary,
	then for $p= \frac{2+n+\alpha}{n}$,
	\begin{eqnarray*}
		I_{\lambda}(u) & \geq & \frac{\left\| \nabla u\right\|_2^2}{2}-\frac{n\mu C_{n,p}}{2+n+\alpha}(\tau)^{\frac{\alpha+2}{n}}\left\| \nabla u\right\|_2^2>0
	\end{eqnarray*}
	for $ \mu<\mu_{*}=\frac{2+n+\alpha}{2nC_{n,p}(\tau)^{\frac{\alpha+2}{n}}}$. Therefore, for each $\tau>0$, there exists a $\mu_{*}$ such that $m_{\lambda}\leq0$ for $\mu \in (0,\mu_{*})$.
	\item[Proof of 3.]	Let $u\in \Gamma$ and  $K>0$ be arbitrary. We define
	$u_K(x)=K^{\frac{n}{2}}u(Kx).$
	Clearly, $u_K\in \Gamma$ and again \eqref{Kargument} holds.
	Since $p>\frac{2+n+\alpha}{n}$, $K^{np-n-\alpha}>K^{2}>K^s$ for large $K$ and we get $\displaystyle -m_{\lambda}\leq \inf_{K}I(u_K)=-\infty$
	which ends the proof.
\end{itemize}
\end{proof}
 \begin{remark}\label{remarklambda=0}
	In case when $\lambda=0$, a simple inspection of the above proof yields that Assertion 1 of Lemma \ref{lemma3.1} holds for  for  $\frac{n+\alpha}{n}<p<\frac{2+n+\alpha}{n}$.
	\end{remark}
	\begin{lemma}\label{lemma3.2}
Assume that $\frac{n+\alpha}{n}<p<\frac{2s+n+\alpha}{n}$, then the minimization problem 
\begin{equation}\label{minI}
	\left\{ \begin{array}{rl}   	
		&  u \in \Gamma   \\
		&  I_{\lambda}(u) = \min\{I_{\lambda}(w)\;:\; w\in \Gamma \}
	\end{array}
	\right.
\end{equation}
has a solution for $\lambda>0$ being small enough.
\end{lemma}

\begin{proof}
Step 1:
Let $\lambda>0$ and $\{u^{\lambda}_k\}$ be a minimizing sequence of $I_{\lambda}$ on $\Gamma$. We claim that $\{u_k^{\lambda}\}$ is bounded in $H^1(\mathbb{R}^n)$.\\
Since $\{I_{\lambda}(u^{\lambda}_k)\}$ converges to $-m_{\lambda}\in(-\infty,0)$ and fixing $M=2|m_\lambda|>0$, one has $|I(u_k^{\lambda})|\leq M$ for every $k\in \mathbb{N}$ (up to an extraction of a subsequence). By \eqref{2.3} and since $\frac{n+\alpha}{n}<p<\frac{2s+n+\alpha}{n}$, we deduce that
$$
\left\| \nabla u_k^{\lambda}\right\|^2_{2}  \leq  \frac{4}{2-np+n+\alpha}\left(M+\frac{a^\kappa}{\kappa}\right)=a'\;\text{ for all }k\in \mathbb{N}.
$$
Further,
\begin{equation*}
	\frac{\lambda [u_k^{\lambda}]^2}{2}  =  I_{\lambda}(u^{\lambda}_k)-\frac{\left\| \nabla u_k^{\lambda}\right\|_2^2}{2}+\frac{\mu}{2p}\mathcal{A}_p(u^{\lambda}_k)
	\leq  M+ \frac{\mu C_{n,p}}{2p}(a')^{\frac{np-n-\alpha}{2}}(\tau)^{\frac{n+\alpha-p(n-2)}{2}}
	=  b.
\end{equation*}
Hence, $\left\|u_k^{\lambda}\right\|^2_{H^1(\mathbb{R}^n)}\leq a'+\tau+2b$ for each $k\in \mathbb{N}$.\\
Step 2:
Taking $\{u_k^{\lambda}\}$ to be the minimizing sequence as deduced in Lemma 2.4 of \cite{jeanjean1997existence}, then by Lemma 2.5 of the same, we can find
a sequence of Lagrange multipliers $\{\Lambda^{\lambda}_k\}$ in $\mathbb{R}$ such that
$H(u^{\lambda}_k)=\tau$ for every $ k\in \mathbb{N}$, $\{I_{\lambda}(u^{\lambda}_k)\}\rightarrow -m_{\lambda}$ and $I_{\lambda}'(u^{\lambda}_k)-\Lambda^{\lambda}_k H'(u^{\lambda}_k)\rightarrow 0$ in $H^{-1}(\mathbb{R}^n)$ as $k\rightarrow \infty $.
Defining $\phi_k(u):=I_{\lambda}(u)-\Lambda^{\lambda}_k H(u)$, for all $u\in H^1(\mathbb{R}^n)$, we have
\begin{equation}\label{phi_n'}
	\phi_k'(u_k^{\lambda})  \rightarrow 0 \text{ as } k \rightarrow\infty.
\end{equation}
Claim 1 :  $\{\Lambda^{\lambda}_k\}$ is bounded and hence converges to $\Lambda^{\lambda}_0 \neq 0$  as $k\rightarrow\infty$, {up to} a subsequence.\\
By \eqref{phi_n'}, we get
$$\left\|\nabla u^{\lambda}_k\right\|_2^2 +\lambda [u^{\lambda}_k]^2-\mu \mathcal{A}_p(u^{\lambda}_k)
= 2\Lambda^\lambda_k\tau +o(1).$$
Thus, from \eqref{G-N} and step 1, $2\tau|\Lambda^\lambda_k| \leq  a'+2b+\mu C_{n,p}(\left\|\nabla u^{\lambda}_k \right\|_2^2)^{\frac{np-n-\alpha}{2}}(\tau)^{\frac{n+\alpha-p(n-2)}{2}} \leq C$.
Therefore, $\{\Lambda^\lambda_k\}$ is bounded and hence it admits a subsequence (denoted by $\{\Lambda^\lambda_k\}$ itself) converging to $\Lambda^\lambda_0\in \mathbb{R}$. Suppose that $\Lambda^\lambda_0=0$. Then,
$\displaystyle \lim_{k\rightarrow \infty}I_{\lambda}'(u_k^{\lambda})=\lim_{k\rightarrow \infty}\Lambda^\lambda_kH'(u^{\lambda}_k)=0 $ and
$$ \left\|\nabla u_k^{\lambda}\right\|_2^2 +\lambda[u^{\lambda}_k]^2 = \mu \mathcal{A}_p(u_k^{\lambda})
+o(1).$$
Thus, 
\begin{eqnarray*}
	-m_{\lambda} & = & \lim_{k\rightarrow \infty}I(u_k^{\lambda})
	=  \frac{1}{2}\left(1-\frac{1}{p}\right)\lim_{k\rightarrow\infty}\left(\left\|\nabla u^{\lambda}_k\right\|_2^2 +\lambda [u^{\lambda}_k]^2\right)
	\geq  0.
\end{eqnarray*}
which contradicts Lemma 3.1 when $\frac{n+\alpha}{n}<p<\frac{2s+n+\alpha}{n}$. Therefore $\Lambda^\lambda_0\neq0$.\\
Claim 2 :  $\Lambda^\lambda_0<0$ and $\{u_k^{\lambda}\}$ satisfies the "asymptotic" Pohozaev-type identity:
$$	  \mu \left(\frac{n+\alpha}{2p}\right)\mathcal{A}_p(u^{\lambda}_k)
+n\Lambda^\lambda_0\tau 
= 	\left(\frac{n-2}{2}\right)\left\|\nabla u^{\lambda}_k\right\|_2^2
+\frac{(n-2s)}{2}\lambda[u^{\lambda}_k]^2
+o(1).$$
We first establish that $\Lambda^\lambda_0<0$, above Pohozaev identity will be tackled later. Since, $\{u^{\lambda}_k\}$ is bounded in $H^1(\mathbb{R}^n)$, there exists a subsequence of $\{u_k^{\lambda}\}$, denoted by $\{u^{\lambda}_k\}$ itself, such that it converges weakly in $H^1(\mathbb{R}^n)$ and $H^s(\mathbb{R}^n)$ to $v^{\lambda}_0\in  H^1(\mathbb{R}^n)$. This implies, for any $v\in C_{c}^\infty(\mathbb{R}^n)$,
\begin{eqnarray*}
	\lim_{k\rightarrow \infty}\ll u^{\lambda}_k,v\gg
	&  = & \lim_{k \rightarrow \infty}\left(\langle u^{\lambda}_k,v\rangle_{H^s(\mathbb{R}^n)} -\langle u^{\lambda}_k,v\rangle_{L^2(\mathbb{R}^n)} \right)  =  \left(\langle v^{\lambda}_0,v\rangle_{H^s(\mathbb{R}^n)} -\langle v^{\lambda}_0,v\rangle_{L^2(\mathbb{R}^n)} \right)\\
	& = &\frac{C(n,s)}{2}\int_{\mathbb{R}^n}\int_{\mathbb{R}^n}\frac{(v^{\lambda}_0(x)-v^{\lambda}_0(y))(v(x)-v(y))}{|x-y|^{n+2s}}dxdy=\ll v^{\lambda}_0,v\gg.
\end{eqnarray*}
In addition, by weak convergence of $\{u^{\lambda}_k\}$ in $H^1(\mathbb{R}^n)$, one has
\begin{eqnarray*}
	\lim_{k\rightarrow \infty}\int_{\mathbb{R}^n}\nabla u^{\lambda}_k \nabla v & = &\int_{\mathbb{R}^n}\nabla v^{\lambda}_0\nabla v.
\end{eqnarray*}
	\noindent	Since, by Proposition 3.1 of \cite{moroz2017guide},
	$\mathcal{A}_p$ is continuously differentiable.
	Thus $\mathcal{A}_p'(u^{\lambda}_k)\rightarrow \mathcal{A}_p'(v^{\lambda}_0)$ [up to a subsequence] in $H^{-1}(\mathbb{R}^n)$. Then, for any $v\in H^1(\mathbb{R}^n)$, we get
	$$\displaystyle \lim_{k\rightarrow \infty}\int_{\mathbb{R}^n}(I_{\alpha}*|u^{\lambda}_k|^p)|u^{\lambda}_k|^{p-2}u^{\lambda}_kvdx=\int_{\mathbb{R}^n}(I_{\alpha}*|v^{\lambda}_0|^p)|v^{\lambda}_0|^{p-2}v^{\lambda}_0vdx.$$
	Using the above result and the fact that $\{\phi_k'(u^{\lambda}_k)\}\rightarrow 0$, we get
	\begin{equation*}
		\int_{\mathbb{R}^n}\nabla v^{\lambda}_0.\nabla vdx+\lambda \ll v^{\lambda}_0,v \gg
		=  \mu\int_{\mathbb{R}^n}(I_\alpha*|v^{\lambda}_0|^p)|v^{\lambda}_0|^{p-2}v^{\lambda}_0vdx
		+2\Lambda^\lambda_0\int_{\mathbb{R}^n}v^{\lambda}_0vdx
	\end{equation*}
	for each $v\in H^1(\mathbb{R}^n)$. Now, if we define
	$\phi_*(u):=I_{\lambda}(u)-\Lambda^\lambda_0 H(u)$
	then, $v^{\lambda}_0$ is a critical point of $\phi_*$, and thus solves
	$$\mathcal{L}(u)= \mu(I_{\alpha}*|u|^p)|u|^{p-2}u+2\lambda_0u\;\;\;\text{in }\mathbb{R}^n.$$
	 Therefore, 
		\begin{equation}\label{R2}
			\left\|\nabla v^{\lambda}_0\right\|_2^2 +\lambda [v^{\lambda}_0]^2 = 2\Lambda^\lambda_0\left\|v^{\lambda}_0\right\|_2^2
			+\mu \mathcal{A}_p(v^\lambda_0).
		\end{equation}
		Further by \autoref{Pohozaev_identity} $v^{\lambda}_0$ satisfies the following Pohozaev type identity
		
		\begin{equation}\label{R1}
			\left(\frac{n-2s}{2}\right)\lambda [v^{\lambda}_0]^2
			+ \left(\frac{n-2}{2}\right)\left\|\nabla v^{\lambda}_0\right\|_2^2
			=  \left(\frac{n+\alpha}{2p}\right)\mu\mathcal{A}_p(v^{\lambda}_0)
			+n\Lambda^\lambda_0\left\|v^{\lambda}_0\right\|_2^2.
		\end{equation}
	
		By \eqref{R1} and \eqref{R2}, we get
		\begin{eqnarray}\label{R3}
			2s\Lambda^\lambda_0H(v^{\lambda}_0) & = & (s-1)\left\|\nabla v^{\lambda}_0\right\|_2^2+\left(\frac{p(n-2s)-n-\alpha}{2p}\right)\mu\mathcal{A}_p(v^{\lambda}_0).
		\end{eqnarray} 
		Now, let $u_{k,1}^{\lambda}:=u^{\lambda}_k-v^{\lambda}_0$. Then $\{u_{k,1}^{\lambda}\}\rightharpoonup 0$ in $ H^1(\mathbb{R}^n)$ and
		\begin{eqnarray*}
			&&	\phi_k(u^{\lambda}_{k,1})-\phi_k(u^{\lambda}_k) =  I_{\lambda}(u^{\lambda}_{k,1})-\Lambda^\lambda_k H(u^{\lambda}_{k,1})-(I_{\lambda}(u^{\lambda}_k)-\Lambda^\lambda_k H(u^{\lambda}_k))\\
			&  & = \frac{1}{2}\int_{\mathbb{R}^n}(|\nabla u^{\lambda}_k-\nabla v^{\lambda}_0|^2-|\nabla u^{\lambda}_k|^2)dx-\Lambda^\lambda_k\int_{\mathbb{R}^n}(|u^{\lambda}_k-v^{\lambda}_0|^2-|u^{\lambda}_k|^2)dx\\
			& & +\frac{\lambda}{2}([u^{\lambda}_k-v^{\lambda}_0]^2-[u^{\lambda}_k]^2)
			-\frac{\mu}{2p}\int_{\mathbb{R}^n}((I_{\alpha}*|u^{\lambda}_k-v^{\lambda}_0|^p)|u^{\lambda}_k-v^{\lambda}_0|^p-(I_{\alpha}*|u^{\lambda}_k|^p)|u^{\lambda}_k|^p)dx.
		\end{eqnarray*}
		By Brezis-Lieb lemma, we have 
		$$\lim_{k\rightarrow \infty}\int_{\mathbb{R}^n}(|u^{\lambda}_k-v^{\lambda}_0|^2-|u^{\lambda}_k|^2+|v^{\lambda}_0|^2)dx=0.$$
		By weak convergence of $\{u^{\lambda}_k\}$ to $v^{\lambda}_0$ in $H^1(\mathbb{R}^n)$
		$$\lim_{k\rightarrow \infty}\int_{\mathbb{R}^n}(|\nabla u^{\lambda}_k-\nabla v^{\lambda}_0|^2-|\nabla u^{\lambda}_k|^2+|\nabla v^{\lambda}_0|^2)dx=0.$$
		By a counterpart of Brezis-Lieb lemma given in \cite[lemma~2.4]{moroz2013groundstates}, we have
		$$\lim_{k\rightarrow \infty}\left(\int_{\mathbb{R}^n}(I_{\alpha}*|u^{\lambda}_k-v^{\lambda}_0|^p)|u^{\lambda}_k-v^{\lambda}_0|^p-\int_{\mathbb{R}^n}(I_{\alpha}*|u^{\lambda}_k|^p)|u^{\lambda}_k|^p+\int_{\mathbb{R}^n}(I_{\alpha}*|v^{\lambda}_0|^p)|v^{\lambda}_0|^p\right)=0,$$
		and 
		\begin{eqnarray*}
			&&     	\lim_{k \rightarrow \infty}\frac{\lambda }{2}([u^{\lambda}_k-v^{\lambda}_0]^2-[u^{\lambda}_k]^2+[v^{\lambda}_0]^2)\\
			&=&\lim_{k\rightarrow \infty}\frac{\lambda}{2}\left(\langle u^{\lambda}_k-v^{\lambda}_0,u^{\lambda}_k-v^{\lambda}_0 \rangle_{H^s(\mathbb{R}^n)}-\langle u^{\lambda}_k-v^{\lambda}_0,u^{\lambda}_k-v^{\lambda}_0 \rangle_{L^2(\mathbb{R}^n)}-\langle u^{\lambda}_k,u^{\lambda}_k \rangle_{H^s(\mathbb{R}^n)}\right.\\
			&&\left.\;\;\;\;\;\;\;\;+\langle u^{\lambda}_k,u^{\lambda}_k \rangle_{L^2(\mathbb{R}^n)}+ \langle v^{\lambda}_0,v^{\lambda}_0 \rangle_{H^s(\mathbb{R}^n)}-\langle v^{\lambda}_0,v^{\lambda}_0 \rangle_{L^2(\mathbb{R}^n)}  \right)=0
		\end{eqnarray*}
		which yields
		\begin{eqnarray}
			\label{main_1}	\phi_k(u^{\lambda}_{k,1}) & = & \phi_k(u^\lambda_k)-\phi_*(v^{\lambda}_0)+o(1),\\
			\label{main_2}   H(u^{\lambda}_{k,1}) & = & H(u^{\lambda}_k)-H(v^{\lambda}_0)+o(1)=\tau-H(v^{\lambda}_0)+o(1),\\
			\label{main_3}    \phi_k'(u^{\lambda}_{k,1}) & = & \phi_k'(u^{\lambda}_k)-\phi_*'(v^{\lambda}_0)+o(1)=o(1)\text{ in } H^{-1}(\mathbb{R}^n).
		\end{eqnarray}
		Taking a test function $\psi\in H^1(\mathbb{R}^n)$, 
		one can proceed similarly as in proof of \cite[Lemma 2.4]{moroz2013groundstates} to show
		\begin{equation*}
			(I_{\alpha}*|u^{\lambda}_k-v^{\lambda}_0|^p)|u^{\lambda}_{k,1}|^{p-2}u^{\lambda}_{k,1}\psi-(I_{\alpha}*|u^{\lambda}_k|^p)|u^{\lambda}_k|^{p-2}u^{\lambda}_k\psi+(I_{\alpha}*|v^{\lambda}_0|^p)|v^{\lambda}_0|^{p-2}v^{\lambda}_0\psi=o(1).
		\end{equation*}
		In particular to get \eqref{main_3}, we combine (strong) convergence of $I_\alpha*(|u^{\lambda}_k-v^{\lambda}_0|^p)-I_\alpha*|u^{\lambda}_k|^p$, $|u^{\lambda}_k|^{p-2}u^{\lambda}_k\psi-|u^{\lambda}_k-v^{\lambda}_0|^{p-2}(u^{\lambda}_k-v^{\lambda}_0)\psi$ respectively to $I_\alpha*|v^{\lambda}_0|^p$ in $L^{\frac{2N}{N-\alpha}}(\mathbb{R}^n)$ and to $|v^{\lambda}_0|^{p-2}v^{\lambda}_0\psi$ in $L^{\frac{2N}{N+\alpha}}(\mathbb{R}^n)$ with weak convergence of $|u^\lambda_{k,1}|^{p-2}u^\lambda_{k,1}\psi$ and $|u^{\lambda}_k|^{p}$ respectively to $0$ and to $|v^{\lambda}_0|^{p}$ (thanks to a.e. convergence) in $L^{\frac{2N}{N+\alpha}}(\mathbb{R}^n)$. Precisely,
		we have:
		\begin{eqnarray*}
			\int_{\mathbb{R}^n}(I_{\alpha}*|u^{\lambda}_k|^p)|u^{\lambda}_k|^{p-2}u^{\lambda}_k\psi-\int_{\mathbb{R}^n}(I_{\alpha}*|u^{\lambda}_k-v^{\lambda}_0|^p)|u^{\lambda}_k-v^{\lambda}_0|^{p-2}(u^{\lambda}_k-v^{\lambda}_0)\psi&&\\
			=  \int_{\mathbb{R}^n}(I_{\alpha}*(|u^{\lambda}_k|^p-|u^{\lambda}_k-v^{\lambda}_0|^p))(|u^{\lambda}_k-v^{\lambda}_0|^{p-2}(u^{\lambda}_k-v^{\lambda}_0)\psi)&&\\
			+\int_{\mathbb{R}^n}(I_{\alpha}*(|u^{\lambda}_k|^p)(|u^{\lambda}_k|^{p-2}u^{\lambda}_k\psi-|u^{\lambda}_k-v_0|^{p-2}(u^{\lambda}_k-v^{\lambda}_0)\psi). &&
		\end{eqnarray*}
		By the strong convergence of $(I_{\alpha}*(|u^{\lambda}_k|^p-|u^{\lambda}_k-v^{\lambda}_0|^p))$ to $I_{\alpha}*|v^{\lambda}_0|^p$ in $L^{\frac{2n}{n-\alpha}}(\mathbb{R}^n)$ and weak convergence $|u^{\lambda}_k-v^{\lambda}_0|^{p-2}(u^{\lambda}_k-v^{\lambda}_0)\psi\rightharpoonup 0$ in $L^{\frac{2n}{n+\alpha}}(\mathbb{R}^n)$, we get
		$$\lim_{k \rightarrow \infty}\int_{\mathbb{R}^n}(I_{\alpha}*(|u^{\lambda}_k|^p-|u^{\lambda}_k-v^{\lambda}_0|^p))|u^{\lambda}_k-v^{\lambda}_0|^{p-2}(u^{\lambda}_k-v^{\lambda}_0)\psi=0,$$
		similarly, by the strong convergence $(|u^{\lambda}_k|^{p-2}u^{\lambda}_k\psi-|u^{\lambda}_k-v^{\lambda}_0|^{p-2}(u^{\lambda}_k-v^{\lambda}_0)\psi)$ to $|v^{\lambda}_0|^{p-2}v^{\lambda}_0\psi$ in $L^{\frac{2n}{n+\alpha}}(\mathbb{R}^n)$ and properties of Riesz potential, we get the strong convergence $I_\alpha *(|u^{\lambda}_k|^{p-2}u^{\lambda}_k\psi-|u^{\lambda}_k-v^{\lambda}_0|^{p-2}(u^{\lambda}_k-v^{\lambda}_0)\psi)$ to $I_\alpha*(|v^{\lambda}_0|^{p-2}v^{\lambda}_0\psi)$ in $L^{\frac{2n}{n-\alpha}}(\mathbb{R}^n)$ which together with weak convergence of $|u^{\lambda}_k|^p$ to $|v^{\lambda}_0|^p$ in $L^{\frac{2n}{n+\alpha}}(\mathbb{R}^n)$ (thanks to a.e. convergence) implies
		\begin{equation*}
			\lim_{k \rightarrow \infty}\int_{\mathbb{R}^n}(I_{\alpha}*(|u^{\lambda}_k|^p)(|u^{\lambda}_k|^{p-2}u^{\lambda}_k\psi-|u^{\lambda}_k-v^{\lambda}_0|^{p-2}(u^{\lambda}_k-v^{\lambda}_0)\psi)
			=\int_{\mathbb{R}^n}(I_{\alpha}*|v^{\lambda}_0|^p)|v^{\lambda}_0|^{p-2}v^{\lambda}_0\psi.
		\end{equation*}
		From above together with $\phi_k'(u^{\lambda}_k)=o(1)$ and $\phi'_*(v^{\lambda}_0)=0$, \eqref{main_3} follows.
		If $\{u^\lambda_{k,1}\}\rightarrow 0$ in $ H^1(\mathbb{R}^n)$, then 
		\begin{equation*}
			\lim_{k \rightarrow \infty}	\left\|\nabla u^{\lambda}_k\right\|_2^2  =  \left\|\nabla v^{\lambda}_0\right\|_2^2\;;\; \Lambda^\lambda_0 \tau=  \lim_{k \rightarrow \infty}\Lambda^\lambda_k \left\| u^{\lambda}_k\right\|_2^2  = \Lambda^\lambda_0\left\|v^{\lambda}_0\right\|_2^2;\;\; 		\lim_{k \rightarrow \infty}\lambda [u^{\lambda}_k]^2  =  \lambda[v^{\lambda}_0]^2
		\end{equation*}
		and by \eqref{G-N}, we have
		$$\mathcal{A}_p(u^\lambda_{k,1})
		\leq C_{n,p}\left(\int_{\mathbb{R}^n}|\nabla u^\lambda_{k,1}|^2dx\right)^{\frac{np-n-\alpha}{2}}\left(\int_{\mathbb{R}^n}|u^\lambda_{k,1}|^2dx\right)^{\frac{n+\alpha-p(n-2)}{2}}\rightarrow 0 \text{ as } k\rightarrow \infty.$$
		Therefore,
		\begin{equation*}
			I_{\lambda}(v^{\lambda}_0)=\phi_*(v^{\lambda}_0)  =  \lim_{k \rightarrow \infty}\phi_k(u^{\lambda}_k)
			=  \lim_{k \rightarrow \infty}(I_{\lambda}(u^{\lambda}_k)-\Lambda^\lambda_k H(u^{\lambda}_k))=-m_{\lambda}-\Lambda^\lambda_0\tau.
		\end{equation*}
		Hence, $v^{\lambda}_0 \in \Gamma$ is the solution to \eqref{minI}. Furthermore, by \eqref{R3}
		\begin{equation*}
			\Lambda^\lambda_0 =  \frac{(s-1)}{2s\tau}\left\|\nabla v^{\lambda}_0\right\|_2^2+\left(\frac{p(n-2s)-n-\alpha}{4ps\tau}\right)\mu \mathcal{A}_p(v^{\lambda}_0)
			<0
		\end{equation*} 
		for $\frac{n+\alpha}{n}<p<\frac{2s+n+\alpha}{n} \text{ and } s \in (0,1).$
		
		Now, we deal with the case where weak convergence does not imply strong convergence in $ H^1(\mathbb{R}^n)$.
		This case can be further devided into the following two subcases:
		$$ \text{Subcase 1 :  } \lim_{k\rightarrow \infty}\left\|u^{\lambda}_k\right\|_2^2 = \left\|v^{\lambda}_0\right\|_2^2\;;\;
		\text{Subcase 2 :  } \lim_{k\rightarrow \infty}\left\|u^{\lambda}_k\right\|_2^2 \neq \left\|v^{\lambda}_0\right\|_2^2.$$
		Subcase 1:
		Clearly $v^{\lambda}_0\neq 0$ as $\left\|v^{\lambda}_0\right\|_2^2=\tau >0$ and hence again by \eqref{R3} $\Lambda^\lambda_0<0$, which gives $\Lambda^\lambda_k<0$ for sufficiently large $k$ and
		$\displaystyle \lim_{k \rightarrow \infty}\left\|u^\lambda_{k,1}\right\|_2^2=0.$
		Also, since $\{u^\lambda_{k,1}\}$ is weakly convergent hence bounded in $ H^1(\mathbb{R}^n)$, then by using \eqref{G-N},  
		$\displaystyle\lim_{k\rightarrow \infty}\mathcal{A}_p(u^\lambda_{k,1})
		=0.$
		This gives
		$$0=\lim_{k\rightarrow \infty}\phi_k'(u^\lambda_{k,1})(u^\lambda_{k,1}) =  \lim_{k\rightarrow \infty}\left(\left\|\nabla u^\lambda_{k,1}\right\|_2^2+\lambda [u^\lambda_{k,1}]^2\right).$$
		Thus, $\{u^\lambda_{k,1}\}\rightarrow 0$ in $ H^1(\mathbb{R}^n)$ and  then as we argue previously, we are done.\\
		Subcase 2:
		Since $\displaystyle \lim_{k\rightarrow \infty}\left\|u^{\lambda}_k\right\|_2^2 \neq \left\|v^\lambda_0\right\|_2^2$, therefore $u^\lambda_{k,1}\nrightarrow 0$ in $L^2(\mathbb{R}^n)$ as $k\rightarrow\infty$ . Then, by Concentration compactness lemma (see\cite[lemma~I.1]{Lions1984concentration}), if $B_1$ is denoting the unit ball of $\mathbb{R}^n$, we can find $\delta_1>0$ and a sequence $\{\zeta_k^1\}$ in $\mathbb{R}^n$ such that  
		\begin{equation}\label{u_n^1}
			\int_{B_1}|u^\lambda_{k,1}(x+\zeta_k^1)|^2dx\geq \delta_1>0
		\end{equation}
		for sufficiently large $k$ (and fixed $\lambda$). Now, we can infer that there exists a such sequence with $|\zeta_k^1|\to\infty$ (as $k\to \infty$) from local compactness in $L^2(\mathbb{R}^n)$ and since $\displaystyle \lim_{k\rightarrow \infty}\left\|u^\lambda_k\right\|_2^2 \neq \left\|v^\lambda_0\right\|_2^2$.
	Define $v^\lambda_{k,1}(x):=u^\lambda_{k,1}(x+\zeta_k^1)$. Clearly, $\{v^\lambda_{k,1}\}$ is bounded in $ H^1(\mathbb{R}^n)$ and hence weakly convergent up to subsequence. Let us denote $v^{\lambda}_1$ the limit of $\{v^\lambda_{k,1}\}$. Then, it implies that $v^\lambda_{k,1}\rightarrow v^{\lambda}_1$ in $L^2_{loc}(\mathbb{R}^n)$ as $k\rightarrow \infty$ and from \eqref{u_n^1}, $v^\lambda_1\neq 0$.
	Using again \eqref{u_n^1} and the fact that $\phi_k'(u^\lambda_{k,1})\rightarrow 0$ in $H^{-1}(\mathbb{R}^n)$ as $k\rightarrow \infty$, we can deduce that 
	\begin{equation*}
		\int_{\mathbb{R}^n}\nabla v^\lambda_1 \nabla wdx-2\Lambda^\lambda_0 \int_{\mathbb{R}^n}v^\lambda_1wdx 
		- \mu \int_{\mathbb{R}^n}(I_{\alpha}*|v^\lambda_1|^p)|v^\lambda_1|^{p-2}v^\lambda_1 wdx=\lambda\ll v^\lambda_1,w \gg
	\end{equation*}
	for every $w\in C_c^{\infty}(\mathbb{R}^n)$. Thus, $v^\lambda_1$ is a critical point of $\phi_*$. Now, as done in \eqref{R3}, 
	\begin{equation*}
		2s\Lambda^\lambda_0H(v^\lambda_1)  =  (s-1)\left\|\nabla v^\lambda_1\right\|_2^2dx+\left(\frac{p(n-2s)-n-\alpha}{2p}\right)\mu\mathcal{A}_p(v^\lambda_1)
	\end{equation*} 
	which gives
	\begin{equation*}
		\Lambda^\lambda_0  =  \frac{(s-1)}{2sH(v_1^\lambda)}\left\| \nabla v^\lambda_1 \right\|_2^2+\left(\frac{p(n-2s)-n-\alpha}{4psH(v_1)}\right)\mu\mathcal{A}_p(v^\lambda_1)
		<0 
	\end{equation*}
	for $\frac{n+\alpha}{n}<p<\frac{2s+n+\alpha}{n} \text{ and } s\in(0,1).$
	Now we show the "asymptotic" Pohozaev type identity. Define $u^\lambda_{k,2}(x):=u^\lambda_{k,1}(x)-v^\lambda_1(x-\zeta_k^1)$, then 
	
	for any $\psi\in C_c^{\infty}(\mathbb{R}^n)$ we have
	$$	\langle u^\lambda_{k,2}, \psi\rangle_{H^1(\mathbb{R}^n)}  =  \langle u^\lambda_{k,1}, \psi\rangle_{H^1(\mathbb{R}^n)} -\langle v^\lambda_1, \psi_k \rangle_{H^1(\mathbb{R}^n)}, $$
	where $\psi_k(x)=\psi(x+\zeta_k^1)$ for all $k\in \mathbb{N}$. Using the weak convergence of $\{u^\lambda_{k,1}\}$ and the fact that $|\zeta_k^1|\rightarrow \infty$ we get $ \langle u^\lambda_{k,2}, \psi \rangle \rightarrow 0$ since $\psi$ has compact support. Hence by
	density of $C_c^{\infty}(\mathbb{R}^n)$ in $H^1(\mathbb{R}^n)$, we get $u^\lambda_{k,2}\rightharpoonup 0$ in $H^1(\mathbb{R}^n)$ as $k\rightarrow \infty$.
	Proceeding, as in \eqref{main_1}, \eqref{main_2} and \eqref{main_3}, we conclude that
	\begin{equation*}
		\phi_k(u^\lambda_{k,2})  =  \phi_k(u^\lambda_k)-\phi_*(v^\lambda_0)-\phi_*(v^\lambda_1)+o(1)\rightarrow -m_\lambda-\Lambda^\lambda_0\tau-\sum_{i=0}^{1}\phi_*(v^\lambda_i) \text{ as } k\rightarrow \infty,
	\end{equation*}
	\begin{equation*}
		H(u^\lambda_{k,2})  =  H(u^\lambda_k)-H(v^\lambda_0)-H(v^\lambda_1)+o(1)=\tau-\sum_{i=0}^{1}H(v^\lambda_i)+o(1)\;,\;\phi_k'(u^\lambda_{k,2})  \rightarrow  0 \text{ as } k\rightarrow \infty
	\end{equation*}
	and since $v^\lambda_1$ is a critical point of $\phi_*$, then similarly as \eqref{R1} we get
	\begin{equation*}
		\left(\frac{n-2s}{2}\right)\lambda [v^\lambda_1]^2
		+ \left(\frac{n-2}{2}\right)\left\|\nabla v^\lambda_1\right\|_2^2 
		=  \left(\frac{n+\alpha}{2p}\right)\mu\mathcal{A}_p(v^\lambda_1)
		+n\Lambda^\lambda_0\left\|v^\lambda_1\right\|_2^2. \nonumber
	\end{equation*}
	Therefore from the above expression and \eqref{R1},
	\begin{equation*}
		\left(\frac{n-2s}{2}\right)\lambda \sum_{i=0}^{1}[v^\lambda_i]^2
		+ \left(\frac{n-2}{2}\right)\sum_{i=0}^{1}\left\|\nabla v^\lambda_i\right\|_2^2
		=  \left(\frac{n+\alpha}{2p}\right)\mu\sum_{i=0}^{1}\mathcal{A}_p(v^\lambda_i)
		+n\Lambda^\lambda_0\sum_{i=0}^{1}\left\|v^\lambda_i\right\|_2^2. \nonumber
	\end{equation*}
	 If $u^\lambda_{k,2}\rightarrow0$ strongly in $H^1(\mathbb{R}^n)$ as $k\rightarrow \infty$, then by definition of $u^\lambda_{k,2}$, we get $u^\lambda_{k,1}-v^\lambda_{1}(.-\zeta_k^1)\rightarrow 0$ as $k\rightarrow \infty$ and hence $v^\lambda_{k,1}-v^\lambda_1=u^\lambda_{k,1}(.+\zeta_k^1)-v^\lambda_1\rightarrow 0$ in $H^1(\mathbb{R}^n)$ as $k\rightarrow\infty$. Thus, by continuity of $\phi_*$
		\begin{eqnarray}\label{phi*(v_1)}
			\phi_*(v^\lambda_1) & = & \lim_{k \rightarrow \infty} \phi_*(v^\lambda_{k,1}) =\lim_{k \rightarrow \infty}\phi_k(u^\lambda_{k,1})=\lim_{k \rightarrow \infty} \phi_k(u^\lambda_k) -\phi_*(v^\lambda_0)\nonumber\\
			& = & -m_{\lambda}-\Lambda^\lambda_0 \tau -\phi_*(v^\lambda_0).
		\end{eqnarray}
	\noindent Now suppose that $\displaystyle \lim_{k\rightarrow \infty}u^\lambda_{k,2}= 0$ in $H^1(\mathbb{R}^n)$ does not hold.  Then, from $\phi_k'(u^\lambda_{k,1})\rightarrow 0$ and $\{u^\lambda_{k,1}\}\rightharpoonup 0$ in $ H^1(\mathbb{R}^n)$ as $k\rightarrow \infty$, we obtain that $\phi_k'(u^\lambda_{k,1})(u^\lambda_{k,1})\rightarrow 0$. Using this and the fact that $v^\lambda_1$ is a critical point of $\phi_*$, we get 
	$$	\phi_*'(v^\lambda_1)(v^\lambda_1) = 0  =  \lim_{k \rightarrow \infty}\phi_k'(u^\lambda_{k,1})(u^\lambda_{k,1})  =  \lim_{k \rightarrow \infty} (I_\lambda'(u^\lambda_{k,1})(u^\lambda_{k,1})-\Lambda^\lambda_k H(u^\lambda_{k,1})(u^\lambda_{k,1})),$$
	that is,
	\begin{eqnarray}\label{2.13}
		&&\left\|\nabla v^\lambda_1\right\|_2^2 +\lambda[v^\lambda_1]^2 -2\Lambda^\lambda_0H(v^\lambda_1)	-\mu\mathcal{A}_p(v^\lambda_1) 
		\nonumber\\
		&& =  \lim_{k \rightarrow \infty}\left(\left\|\nabla u^\lambda_{k,1}\right\|_2^2-2\Lambda^\lambda_k \left\|u^\lambda_{k,1}\right\|_2^2
		+\lambda [u^\lambda_{k,1}]^2 -\mu\mathcal{A}_p(u^\lambda_{k,1})
		\right).
	\end{eqnarray}
	Since $u\mapsto \left\|\nabla u\right\|_2^2$ is a continuous convex functional, then (see Theorem 1.5.3 in \cite{badiale2010semilinear} for instance), it is weakly lower semicontinuous and hence
	\begin{equation*}
		\left\|\nabla v^\lambda_1\right\|_2^2 \leq\liminf \left\|\nabla v^\lambda_{k,1}\right\|_2^2=\liminf \int_{\mathbb{R}^n}|\nabla u^\lambda_{k,1}(x+\zeta_k^1)|^2dx=\liminf \left\|\nabla u^\lambda_{k,1}\right\|_2^2.
	\end{equation*}
	 Now, if $v^\lambda_{k,1}\rightarrow v^\lambda_1$ in $L^2(\mathbb{R}^n)$ as $k\rightarrow \infty$, then 
		\begin{eqnarray*}
			&&	\left|\int_{\mathbb{R}^n}(I_{\alpha}*|v^\lambda_{k,1}-v^\lambda_1|^p)|v^\lambda_{k,1}-v^\lambda_1|^pdx\right| \\
			& \leq & C_{n,p}\left(\int_{\mathbb{R}^n}|\nabla v^\lambda_{k,1}-\nabla v^\lambda_1|^2dx\right)^{\frac{np-n-\alpha}{2}}\left(\int_{\mathbb{R}^n} |v^\lambda_{k,1}-v^\lambda_1|^2dx\right)^{\frac{n+\alpha-p(n-2)}{2}}
			\rightarrow 0 \text{ as } k\rightarrow \infty
		\end{eqnarray*}
		that is $\displaystyle \lim_{k\rightarrow \infty}\mathcal{A}_p(u^\lambda_{k,1})
		= \lim_{k\rightarrow \infty} \mathcal{A}_p(v^\lambda_{k,1})
		= \mathcal{A}_p(v^\lambda_1)
		.$
		Using this in \eqref{2.13}, we get
		\begin{equation*}
			\left\| v^\lambda_1\right\|^2   =\left\|\nabla v^\lambda_1\right\|_2^2 +\lambda  [v^\lambda_1]^2+\left\| v^\lambda_1\right\|_2^2=  \lim_{k \rightarrow \infty}\left(\left\|\nabla u^\lambda_{k,1}\right\|_2^2+ \lambda [u^\lambda_{k,1}]^2+\left\|u^\lambda_{k,1}\right\|_2^2\right)=0
		\end{equation*}
		Since $\left\| u^\lambda_{k,1}\right\|_{H^1(\mathbb{R}^n)}=\left\| v^\lambda_{k,1}\right\|_{H^1(\mathbb{R}^n)}$, it implies $v^\lambda_{k,1}\rightarrow v^\lambda_1$  in $ H^1(\mathbb{R}^n)$ as $k\rightarrow\infty$ but as $\displaystyle \lim_{k\rightarrow \infty}u^\lambda_{k,2}\neq 0$ in $ H^1(\mathbb{R}^n)$ we get a contradiction.
		Hence $v^\lambda_{k,1}\nrightarrow v^\lambda_1$ in $L^2(\mathbb{R}^n)$. 
Thus,
$$0 \neq \lim_{k \rightarrow \infty}\int_{\mathbb{R}^n}|v^\lambda_{k,1}(x)-v^\lambda_1(x)|^2dx=\lim_{k \rightarrow \infty}\int_{\mathbb{R}^n}|u^\lambda_{k,2}(x+\zeta_k^1)|^2dx=\lim_{k \rightarrow \infty}\int_{\mathbb{R}^n}|u^\lambda_{k,2}|^2dx $$
which means $u^\lambda_{k,2}\nrightarrow 0$ in $L^2(\mathbb{R}^n)$ as $k\rightarrow \infty$, then as done earlier, we can find 
$\delta_2>0$ and a sequence $\{\zeta_k^2\}$ in $\mathbb{R}^n$ such that  $|\zeta_k^2|\rightarrow \infty$ as $k\rightarrow \infty$ and
\begin{equation}\label{u_n^2}
	\int_{B_1}|u^\lambda_{k,2}(x+\zeta_k^2)|^2dx\geq \delta_2>0
\end{equation}
for sufficiently large $k$.
Define $v^\lambda_{k,2}(x):=u^\lambda_{k,2}(x+\zeta_k^2)$, clearly $\{v^\lambda_{k,2}\}$ is bounded and hence weakly convergent up to subsequence (still denoted by $\{v^\lambda_{k,2}\}$) in $ H^1(\mathbb{R}^n)$. Let $v^\lambda_2\in  H^1(\mathbb{R}^n)$ be such that
$v^\lambda_{k,2}\rightharpoonup v^\lambda_2 \text{ weakly in }  H^1(\mathbb{R}^n)$ as $k\rightarrow\infty$.
Now, since $\phi_k'(u^\lambda_{k,2})\rightarrow 0$ as $k\rightarrow \infty$, then similarly as done above, $v^\lambda_2$ turns out to be a critical point of $\phi_*$ . We can again define
$u^\lambda_{k,3}(x):=u^\lambda_{k,2}(x)-v^\lambda_2(x-\zeta_k^2).$
Again, if $u^\lambda_{k,3}\rightarrow 0$ in $ H^1(\mathbb{R}^n)$ as $k\rightarrow\infty$, then the process stops and we are done and if $\{u^\lambda_{k,3}\}\nrightarrow 0$ then we need to repeat the process.
Precisely, iterating the above procedure, we can construct the sequences $\{u^\lambda_{k,j}\}_{j\in \mathbb{N}}$ and $\{\zeta_k^j\}_{j\in \mathbb{N}}$ such that $u^\lambda_{k,j+1}(x)  =  u^\lambda_{k,j}(x)-v^\lambda_j(x-\zeta_k^j)$, $v^\lambda_j$ being the weak limit of $\{u^\lambda_{k,j}\}$ in $H^1(\mathbb{R}^n)$ with $v^\lambda_{j}\neq 0$ and:
$$H(u^\lambda_{k,j})  =  H(u^\lambda_k)-\sum_{i=0}^{j-1}H(v^\lambda_i)+o(1);\;\;T(u^\lambda_{k,j}) =T(u^\lambda_k)-\sum_{i=0}^{j-1}T(v^\lambda_i)+o(1),$$
$$[u^\lambda_{k,j}]^2=[u^\lambda_k]^2-\sum_{i=0}^{j-1}[v^\lambda_i]^2+o(1),$$
\begin{equation*}
	\phi_k(u^\lambda_{k,j})  =  \phi_k(u^\lambda_k)-\sum_{i=0}^{j-1}\phi_*(v^\lambda_i)+o(1)
	\rightarrow -m_\lambda-\tau \Lambda^\lambda_0-\sum_{i=0}^{j-1}\phi_*(v^\lambda_i) \text{ as } k\rightarrow \infty.
\end{equation*}
We can see that $v^\lambda_i$ is a critical point of $\phi_* $ for each $i$ and hence it is a solution of
$$\mathcal{L}(u)=\mu(I_{\alpha}*|u|^p)|u|^{p-2}u+2\Lambda^\lambda_0u \;\;\text{ in }\mathbb{R}^n.$$
Now, since for $0\leq i\leq j-1$ $v^\lambda_i$ is a solution of above mentioned equation, we have from \autoref{Pohozaev_identity}:
\begin{equation}\label{2.15}
	\left(\frac{n-2s}{2}\right)\lambda [v^\lambda_i]^2 
	+ \left(\frac{n-2}{2}\right)\left\|\nabla v^\lambda_i\right\|_2^2
	=  \left(\frac{n+\alpha}{2p}\right)\mu\mathcal{A}_p(v^\lambda_i)
	+n\Lambda^\lambda_0\left\|v^\lambda_i\right\|_2^2 
\end{equation}
and
\begin{equation}\label{2.16}
	2s\Lambda^\lambda_0H(v_i^\lambda)  =  (s-1)\left\|\nabla v^\lambda_i\right\|_2^2
	+\left(\frac{p(n-2s)-n-\alpha}{2p}\right)\mu \mathcal{A}_p(v^\lambda_i)
	,\;\;\forall \,0\leq i\leq j-1.
\end{equation} 
As $v^\lambda_i$ is a critical point of $\phi_*$ and $\Lambda^\lambda_0<0$, then by \eqref{G-N} we can deduce the following
\begin{eqnarray}\label{2.17}
	T(v^\lambda_i) & = & \mu\mathcal{A}_p(v^\lambda_i)  
	+2\Lambda^\lambda_0H(v^\lambda_i) 
	<  \mu \mathcal{A}_p(v^\lambda_i)
	\leq  \mu C_{n,p} (T(v^\lambda_i))^{\frac{np-n-\alpha}{2}}(H(v^\lambda_i))^{\frac{n+\alpha-p(n-2)}{2}}\nonumber\\
	\Rightarrow T(v^\lambda_i) & < & K_0(H(v^\lambda_i))^{\frac{n+\alpha-p(n-2)}{2-np+n+\alpha}}, \text{ where } K_0=(\mu C_{n,p})^{\frac{2}{2-np+n+\alpha}}.
\end{eqnarray}
Then, by \eqref{2.16} and \eqref{2.17}
\begin{eqnarray*}
	H(v^\lambda_i) & = & \frac{(s-1)}{2s\Lambda^\lambda_0}\left\|\nabla v^\lambda_i\right\|^2+\mu \frac{(p(n-2s)-n-\alpha)}{4ps\Lambda^\lambda_0}\mathcal{A}_p(v^\lambda_i)
	\\
	& \leq & \frac{(s-1)}{2s\Lambda^\lambda_0}T(v^\lambda_i)+\mu \frac{(p(n-2s)-n-\alpha)}{4ps\Lambda^\lambda_0}C_{n,p}(T(v^\lambda_i))^{\frac{np-n-\alpha}{2}}(H(v^\lambda_i))^{\frac{n+\alpha-p(n-2)}{2}}\\
	& < & \frac{(s-1)}{2s\Lambda^\lambda_0}K_0(H(v^\lambda_i))^{\frac{n+\alpha-p(n-2)}{2-np+n+\alpha}}+C(H(v^\lambda_i))^{\left(\frac{n+\alpha-p(n-2)}{2-np+n+\alpha}\right)\left(\frac{np-n-\alpha}{2}\right)}(H(v^\lambda_i))^{\frac{n+\alpha-p(n-2)}{2}}\\
	&& \text{ with } C=\frac{\mu(p(n-2s)-n-\alpha)}{4ps\Lambda^\lambda_0}C_{n,p}K_0\\
	& = & D(H(v^\lambda_i))^{\frac{n+\alpha-p(n-2)}{2-np+n+\alpha}}, \text{ where } D=\frac{K_0}{2s\Lambda^\lambda_0}\left(s-1+\mu C_{n,p}\frac{p(n-2s)-n-\alpha}{2p}\right)>0.
\end{eqnarray*}
Therefore, for any $i$, $H(v^\lambda_i)\geq \left(\frac{1}{D}\right)^{\beta}\;\;\text{ with }\beta=\frac{2-np+n+\alpha}{2(p-1)},$
which gives for $k$ large
\begin{eqnarray*}
	H(u^\lambda_{k,j}) & = & H(u^\lambda_k)-\sum_{i=0}^{j-1}H(v^\lambda_i)+o(1)
	\leq  \tau-j\left(\frac{1}{D}\right)^{\beta}+\epsilon 
\end{eqnarray*}
for some  $\epsilon >0$ small enough.
Since $H(u^\lambda_{k,j})\geq 0$, j must be finite, which means that the iteration must stop somewhere. Therefore, there exists $K_0 \in\mathbb{N}$ such that $u^\lambda_{k,K_0}\rightarrow 0 \text{ in } H^1(\mathbb{R}^n)$ as $k\rightarrow \infty$.
Since $ H^1(\mathbb{R}^n)$ is continuously embedded in $L^2(\mathbb{R}^n)$, we have
$$o(1)=H(u^\lambda_{k,K_0})=H(u^\lambda_k)-\sum_{i=0}^{K_0-1}H(v^\lambda_i)+o(1).$$ Thus,
\begin{equation}\label{sumH(v_i)}
	\sum_{i=0}^{K_0-1}H(v^\lambda_i)=\tau=\lim_{k \rightarrow \infty}H(u^\lambda_k).
\end{equation}
Similarly, since $ H^1(\mathbb{R}^n)$ is continuously embedded in $H^s(\mathbb{R}^n)$ (for $0<s<1$), we get
\begin{equation}\label{sumT(v_i)}
	\sum_{i=0}^{K_0-1}T(v^\lambda_i)=\lim_{k\rightarrow \infty} T(u^\lambda_k).
\end{equation}
Also, $u^\lambda_{k,K_0}\rightarrow 0$ in $ H^1(\mathbb{R}^n)$ as $k\rightarrow \infty$ implies $\displaystyle o(1)= \phi_k(u^\lambda_{k,K_0}) =  \phi_k(u^\lambda_k)-\sum_{i=0}^{K_0-1}\phi_*(v^\lambda_i)+o(1)$. Thus,
$$ 0  =  \lim_{k\rightarrow \infty}(\phi_k(u^\lambda_k)-\sum_{i=0}^{K_0-1}\phi_*(v^\lambda_i))$$
which gives
$$ \lim_{k\rightarrow \infty}(I_\lambda(u^\lambda_k)-\Lambda^\lambda_k H(u^\lambda_k))  = \sum_{i=0}^{K_0-1}(I_\lambda(v^\lambda_i)-\Lambda^\lambda_0H(v^\lambda_i)),$$ and hence,
\begin{equation*}
	\lim_{k\rightarrow \infty}\left(\frac{1}{2}T(u^\lambda_k)-\Lambda^\lambda_kH(u^\lambda_k)-\frac{\mu}{2p}\mathcal{A}_p(u^\lambda_k)
	\right)
	= \sum_{i=0}^{K_0-1}\left(\frac{1}{2}T(v^\lambda_i)-\Lambda^\lambda_0H(v^\lambda_i)
	-\frac{\mu}{2p}\mathcal{A}_p(v^\lambda_1)
	\right).\\
\end{equation*}
Using \eqref{sumH(v_i)} and \eqref{sumT(v_i)}, we get
\begin{equation}\label{sumI_alpha}
	\lim_{k \rightarrow \infty}\mathcal{A}_p(u^\lambda_k)
	= \sum_{i=0}^{K_0-1}\mathcal{A}_p(v^\lambda_i)
\end{equation}
and since $\displaystyle T(u^\lambda_{k,K_0})-\left\|\nabla u^\lambda_{k,K_0}\right\|_2^2 =  T(u^\lambda_k)-\left\|\nabla u^\lambda_k\right\|_2^2-\sum_{i=0}^{K_0-1}(T(v^\lambda_i)-\left\|\nabla v^\lambda_i\right\|_2^2)+o(1),$  we have
\begin{equation}\label{sum_C(n,s)}
	\lim_{k \rightarrow \infty}\lambda [u^\lambda_k]^2=\sum_{i=0}^{K_0-1}\lambda [v^\lambda_i]^2.\\
\end{equation}
Hence by \eqref{sumT(v_i)} and \eqref{sum_C(n,s)},
$\displaystyle\lim_{k \rightarrow \infty}\left\|\nabla u^\lambda_k\right\|_2^2 = \sum_{i=0}^{K_0-1}\left\|\nabla v^\lambda_i\right\|_2^2.$
By \eqref{2.15}, \eqref{sumH(v_i)}, \eqref{sumI_alpha},\eqref{sumT(v_i)} and \eqref{sum_C(n,s)}, we get
\begin{eqnarray*}
	0 & = & \left(\frac{n-2s}{2}\right)\lambda \sum_{i=0}^{K_0-1}[v^\lambda_i]^2 + \left(\frac{n-2}{2}\right)\sum_{i=0}^{K_0-1}\left\|\nabla v^\lambda_i\right\|_2^2 -n\Lambda^\lambda_0\sum_{i=0}^{K_0-1}\left\|v^\lambda_i\right\|_2^2\\
	&&
	-\left(\frac{n+\alpha}{2p}\right)\mu\sum_{i=0}^{K_0-1}\mathcal{A}_p(v^\lambda_i)
	\\ 
	& = & \left(\frac{n-2s}{2}\right)\lambda \lim_{k \rightarrow \infty}[u^\lambda_k]^2 -n\Lambda^\lambda_0\lim_{k \rightarrow \infty}\left\|u^\lambda_k\right\|_2^2+ \left(\frac{n-2}{2}\right)\lim_{k \rightarrow \infty}\left\|\nabla u^\lambda_k\right\|_2^2 \\
	&&	 -\left(\frac{n+\alpha}{2p}\right)\mu\lim_{k \rightarrow \infty}\mathcal{A}_p(u^\lambda_k).
\end{eqnarray*}
Therefore, the "asymptotic" Pohozaev type identity follows:
\begin{equation*}
	\mu \left(\frac{n+\alpha}{2p}\right)\mathcal{A}_p(u^\lambda_k)
	+n\Lambda^\lambda_0\tau
	= \frac{(n-2s)}{2}\lambda[u^\lambda_k]^2
	+\left(\frac{n-2}{2}\right)\left\|\nabla u^\lambda_k\right\|_2^2
	+o(1)
\end{equation*}
which ends the proof of claim 2.
By \eqref{phi_n'}, we have
\begin{equation}\label{2.22}
	T(u^\lambda_k)=\mu \mathcal{A}_p(u^\lambda_k) 
	+2\tau\Lambda^\lambda_k +o(1),
\end{equation}
equating the value of $\left\|\nabla u^\lambda_k\right\|_2^2$ in above two expressions, we get
\begin{equation}\label{2.23}
	\frac{\mu}{2p}\mathcal{A}_p(u^\lambda_k)
	= \frac{(1-s)}{(n+\alpha-p(n-2))}\lambda[u^\lambda_k]^2 
	+\tau\left(\frac{(n-2)\Lambda^\lambda_k-n\Lambda^\lambda_0}{n+\alpha-p(n-2)}\right) 
	+o(1).
\end{equation}
Using \eqref{2.23} in \eqref{2.22}, we obtain
\begin{equation}\label{2.24}
	T(u^\lambda_k)  =  \frac{2p(1-s)}{n+\alpha-p(n-2)}\lambda[u^\lambda_k]^2
	+\frac{2\tau(-np\Lambda^\lambda_0+(n+\alpha)\Lambda^\lambda_k)}{n+\alpha-p(n-2)}.
\end{equation}
By \eqref{sum_C(n,s)}, \eqref{2.23} and \eqref{2.24}, we can conclude that
\begin{eqnarray}\label{-m}
	-m_\lambda & = & \lim_{k \rightarrow \infty} I_\lambda(u^\lambda_k) =\lim_{k \rightarrow \infty}\left(\frac{T(u^\lambda_k)}{2}-\frac{\mu}{2p}\mathcal{A}_p(u^\lambda_k)
	\right)\nonumber\\
	& = & \lim_{k \rightarrow \infty}\left(\frac{p(1-s) }{(n+\alpha-p(n-2))}\lambda[u^\lambda_k]^2\right.
	+\frac{\tau(-np\Lambda^\lambda_0+(n+\alpha)\Lambda^\lambda_k)}{n+\alpha-p(n-2)}\nonumber \\
	&&-\frac{(1-s) }{(n+\alpha-p(n-2))}\lambda[u^\lambda_k]^2\left.-\tau\left(\frac{(n-2)\Lambda^\lambda_k-n\Lambda^\lambda_0}{n+\alpha-p(n-2)}\right)+o(1)\right)\nonumber\\
	& = & \frac{\tau \Lambda^\lambda_0(np-n-\alpha-2)}{p(n-2)-n-\alpha}+\lim_{k \rightarrow \infty}\frac{(1-s) (p-1)}{(n+\alpha-p(n-2))}\lambda[u^\lambda_k]^2\nonumber\\
	& = & \frac{\tau \Lambda^\lambda_0(np-n-\alpha-2)}{p(n-2)-n-\alpha}+\sum_{i=0}^{K_0-1}\frac{(1-s) (p-1)}{(n+\alpha-p(n-2))}\lambda[v^\lambda_i]^2\nonumber\\
	& \geq & \frac{\tau \Lambda^\lambda_0(np-n-\alpha-2)}{p(n-2)-n-\alpha}
	+\frac{(1-s)(p-1)}{(n+\alpha-p(n-2))}\lambda [v^\lambda _i]^2 \text{ for every } i.
\end{eqnarray}
Claim 3: $H(v^\lambda_i) \geq \tau $ for all $i$.
Suppose that there exists $i$ such that $H(v^\lambda_i)=\gamma<\tau$, then, as $v^\lambda_i$ is a critical point of $\phi_*$ and  by \eqref{2.15}, $v^\lambda_i$ satisfies the following:\\
$H(v^\lambda_i)  =  \gamma\;;\;T(v^\lambda_i) =  \mu \mathcal{A}_p(v^\lambda_i)
+2\gamma\Lambda^\lambda_0,$
\begin{eqnarray*}
	\left(\frac{n+\alpha}{2p}\right)\mu \mathcal{A}_p(v^\lambda_i)
	&=& \left(\frac{n-2}{2}\right)\left\|\nabla v^\lambda_i\right\|_2^2-n\Lambda^\lambda_0\gamma  
	+\left(\frac{n-2s}{2}\right)\lambda [v^\lambda_i]^2.
\end{eqnarray*}
This gives us
$$\frac{\mu}{2p}\mathcal{A}_p(v^\lambda_i)
=  \frac{(1-s)}{(n+\alpha-p(n-2))}\lambda[v^\lambda_i]^2   +\frac{2\Lambda^\lambda_0\gamma}{p(n-2)-n-\alpha}.$$
Therefore,
\begin{equation*}
	I_\lambda(v^\lambda_i) 
	=  \frac{\mu(p-1)}{2p}\mathcal{A}_p(v^\lambda_i)
	+\gamma\Lambda^\lambda_0
	=  \frac{(n+\alpha+2-np)\gamma\Lambda^\lambda_0}{n+\alpha-p(n-2)}+\frac{(1-s)(p-1) }{(n+\alpha-p(n-2))}\lambda[v^\lambda_i]^2.
\end{equation*}
Now, let $\sigma= \tau/\gamma>1$, $$w(x)=\frac{v^\lambda_i(x/\sigma^f)}{\sigma^e}$$
where $e=\frac{2+\alpha}{2(n(p-1)-2-\alpha)}\text{ and }\;f=\frac{p-1}{n(p-1)-2-\alpha},$
then $H(w)=\tau$. Thus, $w\in \Gamma$ and hence by \eqref{-m},
\begin{equation}
	I_\lambda(w)\geq-m_\lambda\geq \frac{\tau \Lambda^\lambda_0(np-n-\alpha-2)}{p(n-2)-n-\alpha}+\frac{(1-s) (p-1)}{(n+\alpha-p(n-2))}\lambda[v^\lambda_i]^2.
\end{equation}
Also,
\begin{eqnarray*}
	I_\lambda(w) & = & \frac{\sigma^{nf-2e-2f}}{2}\left\|\nabla v^\lambda_i\right\|_2^2+\frac{\sigma^{nf-2e-2sf}}{2}\lambda [v^\lambda _i]^2
	-\frac{\mu \sigma^{nf-2ep+f\alpha}}{2p}\mathcal{A}_p(v^\lambda_i)
	\\
	& \leq & \sigma^{nf-2e-2f}I_\lambda (v^\lambda _i)= \sigma^{1-2f}I_\lambda (v^\lambda _i)\\
	& = &  \sigma^{1-2f}\left(\frac{(n+\alpha+2-np)\gamma\Lambda^\lambda_0}{n+\alpha-p(n-2)}+\frac{(1-s)(p-1) }{(n+\alpha-p(n-2))}\lambda [v^\lambda _i]^2\right)\\
	& = & \frac{(n+\alpha+2-np)\Lambda^\lambda_0\tau\sigma^{-2f}}{n+\alpha-p(n-2)}+\frac{\sigma^{1-2f}(1-s)(p-1)}{(n+\alpha-p(n-2))}\lambda [v^\lambda _i]^2.
\end{eqnarray*}
Therefore,
\begin{eqnarray*}
	&&\frac{(n+\alpha+2-np)\Lambda^\lambda_0\tau\sigma^{-2f}}{n+\alpha-p(n-2)}+\frac{\sigma^{1-2f}(1-s)(p-1)}{(n+\alpha-p(n-2))}\lambda [v^\lambda _i]^2\\
	&&\geq I_\lambda(w) \geq -m_\lambda \\
	&& \geq \frac{\tau \Lambda^\lambda_0(np-n-\alpha-2)}{p(n-2)-n-\alpha}+\frac{(1-s) (p-1)}{(n+\alpha-p(n-2))}\lambda [v^\lambda _i]^2,\;\;\forall \lambda>0.
\end{eqnarray*}
Thus, we get
\begin{eqnarray*}
	\frac{(1-s)(\sigma^{1-2f}-1)(p-1)}{(n+\alpha-p(n-2))}\lambda[v^\lambda_i]^2
	& \geq & \frac{\tau \Lambda^\lambda_0(np-n-\alpha-2)}{p(n-2)-n-\alpha}\\ &&-\frac{(n+\alpha+2-np)\Lambda^\lambda_0\tau\sigma^{-2f}}{n+\alpha-p(n-2)},\;\;\forall \lambda>0
\end{eqnarray*}
which implies:
\begin{equation}\label{main}
	\frac{(1-s)(p-1)(\sigma^{1-2f}-1)}{\Lambda^\lambda_0(n+\alpha+2-np)}\lambda [v^{\lambda}_i]^2\leq \tau (1-\sigma^{-2f}),\text{ for all }\lambda>0.
\end{equation}	
 Subclaim 1: There exists $B>0$ such that $[v_i^{\lambda}]^2\leq B$ for any $\lambda\geq 0$.\\
	Since $H^1(\mathbb{R}^n)\hookrightarrow H^s(\mathbb{R}^n)$, there exists a constant $C>0$ such that
	\begin{equation*}
		[v_i^{\lambda}]^2  \leq  C\left(\left\| \nabla v_i^{\lambda}\right\|_2^2+\left\| v_i^{\lambda}\right\|_2^2\right)\leq C\left(\sum_{j=0}^{K_0-1}\left\| \nabla v_j^{\lambda}\right\|_2^2+\tau\right)=C\left(\lim_{k\rightarrow \infty}\left\| \nabla u_k^{\lambda}\right\|_2^2+\tau\right),\\
	\end{equation*}
	now by \eqref{2.3} we have:
	$$\left\| \nabla u \right\|_2^2\leq \frac{4}{n+\alpha+2-np}\left(I_{\lambda}(u)+\frac{a^{\kappa}}{\kappa}\right)\text{ for all }u\in \Gamma \text{ and }\lambda\geq 0,$$
	this gives us
	\begin{eqnarray*}
		[v_1^{\lambda}]^2 & \leq & C\left(\lim_{k\rightarrow \infty}\left(\frac{4}{n+\alpha+2-np}\left(I_{\lambda}(u_k^{\lambda})+\frac{a^{\kappa}}{\kappa}\right)\right)+\tau\right)\\
		& = & C\left(\frac{4}{n+\alpha+2-np}\left(-m_{\lambda}+\frac{a^{\kappa}}{\kappa}\right)+\tau\right)\\
		& \leq & C\tau +\frac{4Ca^{\kappa}}{\kappa(n+\alpha+2-np)} =B \text{ for all } \lambda\geq 0,
	\end{eqnarray*}
	since by Lemma 3.1 and Remark \ref{remarklambda=0}, $-m_{\lambda}:=\inf\{I_{\lambda}(v): v\in \Gamma\}<0$ for any $\lambda\geq 0$. \\
	Subclaim 2:  for $\lambda>0$ small, $\Lambda^\lambda_0\leq -C$ for some $C>0$.\\
	Since $\displaystyle \Lambda^\lambda_0(\lambda)=\lim_{k\rightarrow \infty}\Lambda^\lambda_k$, it suffices to show that there exists $\tilde{C}>0$ such that
	\begin{equation}\label{Claim 2}
		|\Lambda^\lambda_k|\geq \tilde{C}>0 \text{ for sufficiently small }\lambda \text{ and large }k.
	\end{equation}
	Suppose that \eqref{Claim 2} does not hold. Then, there exists 
	a subsequence $\{\Lambda_j\}=\{\Lambda^{\lambda_j}_{k_j}\}$ such that $\lambda_j\rightarrow 0^+$, $k_j\rightarrow \infty$ and $\{\Lambda_j\}\rightarrow 0^-$ as $j\rightarrow\infty$. By sake of simplicity, let us denote $u_j^{\lambda_j}$ by $u_j$. As done in step 1 of the proof together with the fact that $\lambda\mapsto m_\lambda$ is nondecreasing (which implies that the bound $M=2|m_{\lambda_j}|$ is uniformly bounded with respect to $j$), one can see that the sequence $\{u_j\}$ is uniformly bounded in $H^1(\mathbb{R}^n)$. Then $\displaystyle \lim_{j\rightarrow \infty}I'_{\lambda_j}(u_{j})=\lim_{j\rightarrow \infty}\Lambda_jH'(u_j)=0$ which implies:
	$$\left\|\nabla u_j \right\|_2^2+\lambda_j[u_j]^2=\mu\mathcal{A}_p(u_j)+o(1)$$
	and hence, 
	\begin{eqnarray*}
		\liminf_{j\to\infty}(-m_{\lambda_j}) & = & \liminf_{j\rightarrow \infty}I_{\lambda_j}(u_j) =\frac{(p-1)}{2p}\liminf_{j\rightarrow \infty}\left(\left\|\nabla u_j \right\|_2^2+\lambda_j[u_j]^2\right)\geq 0.
	\end{eqnarray*}
	But since $\lambda_j$ approaches 0 as $j\rightarrow\infty$ we can find $\tilde{\lambda}>\lambda_j$ for all $j$ and by monotonicity of $\lambda\mapsto m_\lambda$ together with Lemma \ref{lemma3.1} one has
	$$\liminf_{j\to\infty}(-m_{\lambda_j})\leq-m_{\tilde{\lambda}}<0.$$
	Therefore, we get a contradiction and Subclaim 2 holds. Now, taking $\lambda$ as small as possible, inequality \eqref{main} yields
	$$0\leq \tau(1-\sigma^{-2f})$$
	and hence $1\geq \sigma^{-2f}$, but since $\sigma>1$ and $f<0$, we get a contradiction,
thus, $H(v^\lambda _i)\geq \tau$, for every $i$. Also, by \eqref{sumH(v_i)}, $\displaystyle \sum_{i=0}^{K_0-1} H(v^\lambda _i)=\tau$.
Now if $v^\lambda _0\neq 0 $, then $H(v_0)=\tau$, and we are done as we circle back to the subcase 1.
Let $v^\lambda _0=0$, then we must have
$v^\lambda _1\neq 0$. Thus
$H(v^\lambda _1)=\tau \Rightarrow v^\lambda _1\in \Gamma$ and $v^\lambda _j=0$ for all $ j\geq 2$. This implies that the iteration stopped at $K_0=2$ which means $u^\lambda_{k,2}\rightarrow 0$ as $k\rightarrow \infty$ and by \eqref{phi*(v_1)}, $\phi_*(v^\lambda _1)=-m_\lambda-\Lambda^\lambda_0\tau$, thus $I_\lambda (v^\lambda _1)=-m_\lambda $. Hence, we are done.
\end{proof}

\begin{lemma}
Assume that $\frac{n+\alpha}{n}<p<\frac{2s+n+\alpha}{n}$, and $\lambda>0$ is sufficiently small. Then the minimization problem \eqref{minS} has at least a solution $u$. Moreover,
the solution of \eqref{minS} will solve \eqref{prob}.
\end{lemma}
\begin{proof}
Claim 1: \eqref{minI} and \eqref{minS} are equivalent.
Since for $u\in \Gamma \;S_\lambda(u)=I_\lambda(u)+\frac{\tau}{2}$, then,
clearly, if $u\in \Gamma$ is a solution of \eqref{minI},  it must be the solution of \eqref{minS} as well and vice versa. Thus \eqref{minI} and \eqref{minS} are equivalent. Hence by Lemma 3.2, the problem \eqref{minS} has a solution $u$.
Now, let $u\in\Gamma$ be a solution of \eqref{minS}.
Claim 2: $u$ is a solution of \eqref{prob}.\\
Define $h:=\inf\{S_\lambda(v):v\in \Gamma\}.$
Let
$u_{K}(x)=K^{\frac{n}{2}}u(Kx),$
then, clearly $u_K\in \Gamma$. Since $u_1=u$ is a solution of \eqref{minS}, then
$\frac{d}{dK}S_\lambda(u_K)|_{K=1}=0$ implying 
Pohozaev identity:
\begin{equation*}
	\frac{\mu}{2p}\mathcal{A}_p(u)
	=  \frac{\left\| \nabla u \right\|_2^2}{np-n-\alpha}
	+\frac{s\lambda [u]^2 }{(np-n-\alpha)}.
\end{equation*}
Also, by the method of Lagrange multipliers, there exists $\delta>0$ such that $u$ is a solution of
$$\mathcal{L}(u)+\delta u=\mu(I_{\alpha}*|u|^p)|u|^{p-2}u\;\;\text{ in }\mathbb{R}^n.$$
This gives us
\begin{eqnarray*}
	\delta\tau 
	& = & \left(\frac{n+\alpha-p(n-2)}{np-n-\alpha}\right)\left\|\nabla u\right\|_2^2
	+\lambda \left(\frac{n+\alpha-p(n-2s)}{np-n-\alpha}\right)[u]^2>0
\end{eqnarray*}
for $\frac{n+\alpha}{n}<p<\frac{2s+n+\alpha}{n}$,
thus $\delta>0$.
Now define $v$, such that  for any $x\in\mathbb{R}^n$
$u(x)=\delta^{\frac{2+\alpha}{4(p-1)}}v(\delta^{\frac{1}{2}}x).$
Then, we can deduce that $v$ is a solution of 
\begin{equation}\label{Prob_v}
	-\Delta v+\delta^{s-1}\lambda(-\Delta)^sv+v= \mu (I_{\alpha}*|v|^{p})|v|^{p-2}v\;\;\text{ in }\mathbb{R}^n.
\end{equation} 
This gives 
$$\mu\mathcal{A}_p(v)
= \left\|\nabla v\right\|_2^2+\delta^{s-1}\lambda [v]^2+\left\|v\right\|_2^2$$
and if $F$ is the variational functional corresponding to \eqref{Prob_v}, then $\frac{d}{dK}F(v(x/K))|_{K=1}  =  0$. Thus, again Pohozaev identity holds:
\begin{equation*}
	\mu \mathcal{A}_p(v)
	= \frac{p(n-2)}{n+\alpha}\left\|\nabla v\right\|_2^2+\frac{np}{n+\alpha}\left\|v\right\|_2^2
	+\frac{\delta^{s-1}(n-2s)p}{(n+\alpha)}\lambda[v]^2.
\end{equation*}
Equating above equations, we get
\begin{eqnarray*}
	\left\|\nabla v\right\|_2^2 & = & \delta^{s-1}\lambda \left(\frac{(n-2s)p-n-\alpha}{(n+\alpha-p(n-2))}\right)[v]^2
	+ \left(\frac{np-n-\alpha}{n+\alpha-p(n-2)}\right)\left\|v\right\|_2^2.
\end{eqnarray*}
Taking $A=\frac{2+\alpha+n-np}{2(p-1)},\;B=\frac{n+\alpha-p(n-2)}{2(p-1)} \text{ and } C=\frac{2+\alpha-(n-2s)(p-1)}{2(p-1)},$
we get
\begin{eqnarray*}
	S_\lambda(v) & = & \frac{1}{2}\left( \delta^{s-1}\lambda \left(\frac{(n-2s)p-n-\alpha}{(n+\alpha-p(n-2))}\right)[v]^2
	+ \left(\frac{np-n-\alpha}{n+\alpha-p(n-2)}\right)\left\|v\right\|_2^2\right)\\
	&&+\frac{\lambda }{2}[v]^2
	+\frac{\left\| v \right\|_2^2}{2}-\frac{1}{2p}\left(\left\|\nabla v\right\|_2^2+\delta^{s-1}\lambda[v]^2
	+ \left\|v\right\|_2^2\right)\\
	& = & \frac{\lambda }{2}\left(\delta^{s-1}\left(\frac{(n-2s)p-n-\alpha}{(n+\alpha-p(n-2))}\right)+1-\frac{\delta^{s-1}}{p}\right)[v]^2\\
	&& +\frac{1}{2}\left(\frac{np-n-\alpha}{n+\alpha-p(n-2)}-\frac{1}{p}+1\right)\left\|v\right\|_2^2      
	-\frac{1}{2p}\left( \left(\frac{np-n-\alpha}{n+\alpha-p(n-2)}\right)\left\|v\right\|_2^2\right.\\
	&& \left. +\delta^{s-1}\lambda \left(\frac{(n-2s)p-n-\alpha}{(n+\alpha-p(n-2))}\right)[v]^2\right)\\
	& = &  \left(1-\frac{\delta^{s-1}C}{B}\right)\frac{\lambda }{2}[v]^2+\frac{p-1}{n+\alpha-p(n-2)}\left\|v\right\|_2^2.
\end{eqnarray*}
Also, we can deduce that 
\begin{equation}
	\left\|\nabla u\right\|_2^2  =  \delta^B\left\|\nabla v\right\|_2^2\; , \;	\tau = \left\|u\right\|_2^2  = \delta^{A}\left\|v\right\|_2^2,\;\;	[u]^2=\delta^C[v]^2\;\mbox{and}\;		
	\mathcal{A}_p(u)	=  \delta^{B}\mathcal{A}_p(v)
	.
\end{equation}
Therefore,
\begin{equation}\label{S(v)}
	S_\lambda (v)=\frac{\tau\delta^{-A}}{2B}+\left(1-\frac{\delta^{s-1}C}{B}\right)\frac{\lambda }{2}[v]^2.
\end{equation}
Now,
\begin{eqnarray*}
	S_\lambda(u) & = & \frac{\delta^B}{2}\left\|\nabla v\right\|_2^2+\frac{\delta^A}{2}\left\|v\right\|_2^2+\frac{\delta^C }{2} \lambda[v]^2
	-\frac{\delta^{B}\mu}{2p}\mathcal{A}_p(v)
	\\
	& = & \delta^B S(v)+\frac{(\delta^A-\delta^B)}{2}\left\|v\right\|_2^2 +\frac{(\delta^C-\delta^B) }{2} \lambda[v]^2\\
	&=& \delta^B S(v)+\frac{\delta^A(1-\delta)}{2}\delta^{-A}\tau+\frac{(\delta^C-\delta^B)\lambda}{2}[v]^2.
\end{eqnarray*}
Using \eqref{S(v)}, we get
\begin{eqnarray*}
	S_\lambda(u) & = & \delta^B S(v)+\frac{(1-\delta)}{2}\left(\lambda (C\delta^{s-1+A}-B\delta^A)[v]^2\right.
	\left.+2BS(v)\delta^A\right)
	+\frac{(\delta^C-\delta^B)\lambda }{2}[v]^2 \\
	& = & (-A\delta^B+B\delta^A)S(v)+\lambda K_\delta,
\end{eqnarray*}
where
$K_\delta=\frac{1}{2}((\delta^A-\delta^B)(C\delta^{s-1}-B)+(\delta^C-\delta^B))[v]^2.$
Thus, we have $$\frac{S(u)-\lambda K_\delta}{S(v)}=-A\delta^B+B\delta^A.$$
Case 1: 
$\frac{S_\lambda(u)-\lambda K_\delta}{S_\lambda(v)}=-A\delta^B+B\delta^A>1$.
For all $t>0$, let us define
$$f(t)=-At^{B}+Bt^A-\left(\frac{S_\lambda(u)-\lambda K_\delta}{S_\lambda(v)}\right),$$
then $f(\delta)=0,$
$f(1)=-A+B-\left(\frac{S_\lambda(u)-\lambda K_\delta}{S_\lambda(v)}\right)=1-\left(\frac{S_\lambda(u)-\lambda K_\delta}{S_\lambda(v)}\right)<0,$
and $f'(t)=-ABt^{B-1}+ABt^{A-1}=ABt^{A-1}(1-t).$
Thus, $f$ is increasing in $(0,1)$, decreasing in $(1,\infty)$ and attains its maximum at $t=1$. Since $f(1)<0$, this implies that there does not exist any $\delta>0$ such that $f(\delta)=0$.\\
Case 2:   $0\leq\frac{S_\lambda(u)-\lambda K_\delta}{S_\lambda(v)}=-A\delta^B+B\delta^A\leq1$.
Again, we will use the same function $f$. In this case
$f(0)=-\left(\frac{S_\lambda(u)-\lambda K_\delta}{S_\lambda(v)}\right)<0,
f(1)=-A+B-\left(\frac{S_\lambda(u)-\lambda K_\delta}{S_\lambda(v)}\right)=1-\left(\frac{S_\lambda(u)-\lambda K_\delta}{S_\lambda(v)}\right)\geq 0.$
If $f(1)>0$, then the two distinct roots of $f$ will give two distinct values of $\delta$. But since
$\tau=\delta^{A}\left\|v\right\|_2^2$ is unique, we get $f(1)=0$. Thus $\delta=1$.\\
Case 3: $\frac{S_\lambda(u)-\lambda K_\delta}{S_\lambda(v)}=-A\delta^B+B\delta^A<0$.
In this case, $f$ will have an unique root $\delta>1$. Then since $v$ is a solution of \eqref{Prob_v}, we get
\begin{equation*}
	S_\lambda(v)  >  \frac{\left\|\nabla v\right\|_2^2}{2} +\frac{\left\|v\right\|_2^2}{2}+\frac{\delta^{s-1}}{2}\lambda[v]^2
	-\frac{\mu}{2p}\mathcal{A}_p(v)
	=  \frac{\mu(p-1)}{2p} \mathcal{A}_p(v)
	>0.
\end{equation*}
Thus, $S_\lambda(v)>0$ and hence $S_\lambda(u)-\lambda K_{\delta}<0$. Now, taking $C_{\delta}=((\delta^A-\delta^B)(C\delta^{s-1}-B)+(\delta^C-\delta^B))\delta^{-C}$
we get
\begin{eqnarray*}
	0&>&  S_\lambda(u)-\lambda K_{\delta}\\
	& = & \frac{\left\|\nabla u\right\|_2^2}{2}+\frac{\tau}{2}-\frac{\mu}{2p}\mathcal{A}_p(u)
	+\frac{\lambda [u]^2 }{2}
	-\frac{\lambda }{2}((\delta^A-\delta^B)(C\delta^{s-1}-B)+(\delta^C-\delta^B))[v]^2\\
	& = & \frac{\left\|\nabla u\right\|_2^2}{2}+\frac{\tau}{2}-\frac{\mu}{2p}\mathcal{A}_p(u)
	+\frac{\lambda }{2}(1-C_{\delta})[u]^2
	>  \frac{\left\|\nabla u\right\|_2^2}{2}+\frac{\tau}{2}-\frac{\mu}{2p}\mathcal{A}_p(u)
	-C_{\delta}\frac{\lambda }{2}[u]^2,
\end{eqnarray*}
that is,
\begin{eqnarray*}
	C_{\delta}\frac{\lambda[u]^2}{2} & > & \frac{\left\|\nabla u\right\|_2^2}{2}+\frac{\tau}{2}-\frac{\mu}{2p}\mathcal{A}_p(u)
	.
\end{eqnarray*}
Using \eqref{G-N} and the Young's inequality for $\frac{n+\alpha}{n}<p<\frac{2s+n+\alpha}{n}<\frac{2+n+\alpha}{n}$, we deduce that
\begin{eqnarray*}
	C_{\delta}\frac{\lambda [u]^2}{2}
	& > & \frac{\left\|\nabla u\right\|_2^2}{2}+\frac{\tau}{2}-C_{n,p}\frac{\mu}{2p}\left(\left\|\nabla u\right\|_2^2\right)^{\frac{np-n-\alpha}{2}}(\tau)^{\frac{n+\alpha-p(n-2)}{2}}\\
	& \geq & \frac{\left\|\nabla u\right\|_2^2}{2}-\left(\frac{np-n-\alpha}{2}\right)\left(\frac{1}{2}\left(\left\|\nabla u\right\|_2^2\right)^{\frac{np-n-\alpha}{2}}\right)^{\frac{2}{np-n-\alpha}}\\
	& & -\left(\frac{n+\alpha+2-np}{2}\right)\left(\frac{C_{n,p}\mu}{p}(\tau)^{\frac{n+\alpha-p(n-2)}{2}}\right)^{\frac{2}{n+\alpha+2-np}}+\frac{\tau}{2}\\
	& \geq & \left(\frac{1}{2}-\frac{np-n-\alpha}{4}\right)\left\|\nabla u\right\|_2^2\\
	&& +\frac{\tau}{2}\left(1-(n+\alpha+2-np)\left(\frac{C_{n,p}\mu}{p}\right)^{\frac{2}{n+\alpha+2-np}}(\tau)^{\frac{2(p-1)}{n+\alpha+2-np}}\right)\\
	& \geq & \frac{\tau}{2}\left(1-(n+\alpha+2-np)\left(\frac{C_{n,p}\mu}{p}\right)^{\frac{2}{n+\alpha+2-np}}(\tau)^{\frac{2(p-1)}{n+\alpha+2-np}}\right).
\end{eqnarray*}
Let $G=(n+\alpha+2-np)\left(\frac{C_{n,p}\mu}{p}\right)^{\frac{2}{n+\alpha+2-np}}(\tau)^{\frac{2(p-1)}{n+\alpha+2-np}}$.
Then, we have $$\frac{\tau}{2}(1-G)<C_{\delta}\frac{\lambda [u]^2}{2},\;\;\;\forall \lambda>0.$$
Since the left hand side is independent of the paramter $\lambda$ which can be as small as we want, we must have
$\frac{\tau}{2}(1-G)\leq0\;\;\forall\;\mu>0.$
Netherveless, for $$0<\mu<\mu^*=\frac{p(\tau)^{(1-p)}}{C_{n,p}(n+\alpha+2-np)^{\frac{n+\alpha+2-np}{2}}},$$
we get $G<1$, which gives a contradiction.
Hence, we are left with only one possiblity, $\delta=1$.
Therefore, $u=v$ is a solution of \eqref{prob}.
\end{proof}
\begin{myproof}{Theorem}{\ref{Theorem 1.3} and \ref{Theorem 1.4}}
By Lemma 3.2, we know that for $\frac{n+\alpha}{n}<p<\frac{2s+n+\alpha}{n}$, there exists a solution of \eqref{minI}. Since, by step 1 of Lemma 3.3, \eqref{minI} and \eqref{minS} are equivalent, \eqref{minS} has a solution and hence by Lemma 3.3, \eqref{prob} has a solution. Thus, we are done with Theorem 1.3.
As the minimizer of S on $\Gamma$ belongs to $A$, thus $A$ is nonempty.\\
Now, let $u$ be a solution of \eqref{minS}, and hence solves \eqref{prob}. Then as $u$ becomes a critical point of $S$, we get
\begin{equation}
	H(u)  =  \tau\;;\;T(u)+\tau  =  \mu\mathcal{A}_p(u)
\end{equation}
and
\begin{equation*}
	\mu\mathcal{A}_p(u)
	=  \lambda\left(\frac{n-2s}{2}\right)[u]^2+\left(\frac{n-2}{2}\right)\left\|\nabla u\right\|_2^2
	+\frac{n\tau}{2}.
\end{equation*}
From the above equations, we can deduce that
\begin{eqnarray*}
	\frac{\mu}{2p}\mathcal{A}_p(u)
	& = &\frac{\tau}{n+\alpha-p(n-2)}
	+\frac{\lambda(1-s)[u]^2}{(n+\alpha-p(n-2))}
\end{eqnarray*}
and
$$\frac{\mu}{2p}\mathcal{A}_p(u)
=\frac{\tau s}{n+\alpha-p(n-2s)}+\frac{(s-1)}{n+\alpha-p(n-2s)}\left\| \nabla u\right\|_2^2$$
which give
\begin{eqnarray*}
	S_\lambda(u) & = &  \frac{\mu(p-1)}{2p}\mathcal{A}_p(u)
	=  \frac{\tau(p-1)}{n+\alpha-p(n-2)}+\frac{\lambda(1-s)(p-1)[u]^2}{(n+\alpha-p(n-2))}
\end{eqnarray*}
implying that
\begin{eqnarray*}
	I_\lambda(u) & = &\frac{\tau(np-n-\alpha-2)}{2(n+\alpha-p(n-2))}+ \frac{\lambda(1-s)(p-1)[u]^2}{(n+\alpha-p(n-2))}
\end{eqnarray*}
and
$$S_\lambda(u)=\frac{\tau s(p-1)}{n+\alpha-p(n-2s)}+\frac{(s-1)(p-1)}{n+\alpha-p(n-2s)}\left\| \nabla u\right\|_2^2.$$
Now, let $w\in G$ be arbitrary, this implies $S_\lambda(w)=l$ and $w$ is a solution of 
$$\mathcal{L}(w)+w=\mu(I_{\alpha}*|w|^p)|w|^{p-2}w\;\;\text{ in }\;\mathbb{R}^n.$$
Then by similar calculations, as done above, we get
\begin{equation}
	S_\lambda (w) = \frac{s(p-1)\gamma}{n+\alpha-p(n-2s)}+\frac{(s-1)(p-1)}{n+\alpha-p(n-2s)}\left\|\nabla w\right\|_2^2,
\end{equation}
$$ S_\lambda (w)=\frac{\gamma(p-1)}{n+\alpha-p(n-2)}+\frac{\lambda (p-1)(1-s)[w]^2}{(n+\alpha-p(n-2))}$$
and 
\begin{equation}
	I_\lambda (w) =\frac{\gamma(np-n-\alpha-2)}{2(n+\alpha-p(n-2))}+ \frac{\lambda(1-s)(p-1)[w]^2 }{(n+\alpha-p(n-2))}.
\end{equation}
Now, $\gamma\geq \tau$. Indeed, if $\gamma< \tau$, then $\sigma=\tau/\gamma>1$. Defining $v$ such that
$w(x)=\sigma^av(\sigma^bx)$
where
$a=\frac{2+\alpha}{2(n(p-1)-2-\alpha)}\text{ and }b=\frac{p-1}{n(p-1)-2-\alpha},$
one has $H(v)=\tau$. Thus, $v\in \Gamma$. Since $u$ minimizes $I_\lambda $ on $\Gamma$, we have
\begin{eqnarray*}
	I_\lambda (u) & \leq & I_\lambda (v)=\frac{\sigma^{nb-2a-2b}}{2}\left\|\nabla w\right\|_2^2-\frac{\mu \sigma^{nb-2ap+b\alpha}}{2p}\mathcal{A}_p(w)
	+\frac{\sigma^{-2a-2bs+nb}}{2}\lambda [w]^2\\
	& \leq & \sigma^{1-2b}I(w)=\sigma^{1-2b}\left(\frac{\gamma(np-n-\alpha-2)}{2(n+\alpha-p(n-2))}
	+ \frac{\lambda(1-s)(p-1)[w]^2}{(n+\alpha-p(n-2))}\right).
\end{eqnarray*} 
Thus, we have
\begin{eqnarray*}
	&& \frac{(np-n-\alpha-2)\tau}{2(n+\alpha-p(n-2))}+\frac{\lambda (1-s)(p-1)[u]^2}{(n+\alpha-p(n-2))}\\
	&&=  I_\lambda (u)\leq I_\lambda (v)\\
	&&\leq  \sigma^{1-2b}\left(\frac{\gamma(np-n-\alpha-2)}{2(n+\alpha-p(n-2))}
	+ \frac{\lambda(1-s)(p-1)[w]^2}{(n+\alpha-p(n-2))}\right).      		
\end{eqnarray*}
Therefore,
\begin{eqnarray*}
	\frac{(np-n-\alpha-2)\tau}{2(n+\alpha-p(n-2))}-\sigma^{1-2b}\left(\frac{\gamma(np-n-\alpha-2)}{2(n+\alpha-p(n-2))}\right) &\leq & \lambda\left(\frac{(1-s)(p-1)}{(n+\alpha-p(n-2))}\right) \left( \sigma^{1-2b}[w]^2\right.\\ && \left.-[u]^2\right)\;\;\forall\;\lambda>0.
\end{eqnarray*}
As the left hand side of the above expression is independent of the parameter $\lambda$, we get
$$\frac{(np-n-\alpha-2)\tau}{2(n+\alpha-p(n-2))}\leq\sigma^{1-2b}\frac{\gamma(np-n-\alpha-2)}{2(n+\alpha-p(n-2))}=\sigma^{-2b}\frac{\tau(np-n-\alpha-2)}{2(n+\alpha-p(n-2))}.$$
Thus, for $\frac{n+\alpha}{n}<p<\frac{2s+n+\alpha}{n}$ and $s\in (0,1)$, we get
$\sigma^{-2b}\leq 1.$
But since $\sigma>1$ and $b<0$ for $\frac{n+\alpha}{n}<p<\frac{2s+n+\alpha}{n}$, then  $\sigma^{-2b}$ must be greater than 1. Hence, by contradiction, $\sigma \leq 1$ and consequently $\gamma\geq\tau.$
Now, if $\gamma=\tau$, then $w\in \Gamma$ and hence $S(w)\geq S(u)$. For $\gamma>\tau$ we have
\begin{eqnarray*}
	S_\lambda (w)-S_\lambda (u)
	& = &\frac{\gamma(p-1)}{n+\alpha-p(n-2)}+\frac{\lambda (p-1)(1-s)}{(n+\alpha-p(n-2))}[w]^2\\
	&& -\frac{\tau(p-1)}{n+\alpha-p(n-2)}-\frac{\lambda(1-s)(p-1)}{(n+\alpha-p(n-2))}[u]^2\\
	& = & \frac{(\gamma-\tau)(p-1)}{n+\alpha-p(n-2)}+\frac{\lambda(1-s)(p-1)}{(n+\alpha-p(n-2))}\left([w]^2\right. \left.-[u]^2\right),\;\;\;\forall \lambda>0.
\end{eqnarray*}
As $(\gamma-\tau)>0$, no matter how small is $\lambda$, we obtain $S(w)-S(u)\geq 0$ for $\frac{n+\alpha}{n}<p<\frac{2s+n+\alpha}{n}$. Therefore, $S_\lambda (w)\geq S_\lambda (u)$ which gives
$l\geq S_\lambda (u).$
Also, since $u\in A$, we get
$S_\lambda (u)\geq l.$
Thus, $S_\lambda (w)=l=S_\lambda (u)$ and hence $u\in G$. This implies that $G$ is non empty. \\
Conversely, let $w\in G$, then by above proof $H(w)=\gamma\geq \tau$ and $S_\lambda (w)=S_\lambda (u)$, where $u$ is a solution of \eqref{minS}.
Then, as both $u$ and $w$ are in $A$, we get
\begin{equation*}
	\frac{\mu(p-1)}{2p}\mathcal{A}_p(u)
	=S_\lambda (u)  = S_\lambda (w) =\frac{\mu(p-1)}{2p}\mathcal{A}_p(w)
	.
\end{equation*}
We claim that $\gamma=\tau$. 
Let if possible $\gamma>\tau$. 
Now, since $S_\lambda (u)=S_\lambda (w)$, then
$$0>\frac{s(p-1)(\tau-\gamma)}{n+\alpha-p(n-2s)}=\frac{(s-1)(p-1)}{n+\alpha-p(n-2s)}\left(\left\|\nabla w\right\|_2^2-\left\|\nabla u\right\|_2^2\right).$$
Thus, $\left\|\nabla w\right\|_2^2>\left\|\nabla u\right\|_2^2.$
For  sufficiently small $\lambda>0$, we then get
\begin{eqnarray*}
	S_\lambda (w)-S_\lambda (u) & = & \frac{1}{2}\left(\left\|\nabla w\right\|_2^2-\left\|\nabla u\right\|_2^2\right) +\lambda\left([w]^2\right.\left.-[u]^2\right)+\frac{(\gamma-\tau)}{2}\\
	& \geq &\frac{1}{2}\left(\left\|\nabla w\right\|_2^2dx-\left\|\nabla u\right\|_2^2\right)>0.
\end{eqnarray*}
Therefore, $S_\lambda (w)>S_\lambda (u)$ which is a contradiction to the fact that $S_\lambda (u)=S_\lambda (w)$. Consequently, $\gamma=\tau$ and hence $w\in \Gamma$. Thus, $w$ is a solution of \eqref{minS}.
Therefore, $u\in G$ if and only if $u$ solve \eqref{minS}.
Now, if $u\in G$ then u solves \eqref{minS} and hence by Lemma 3.3, $u$ is a solution of \eqref{prob}. Conversely, if $u$ solves \eqref{prob}, then $u\in \Gamma$ and it is a solution of $$\mathcal{L}(u)+u=\mu(I_{\alpha}*|u|^p)|u|^{p-2}u\;\;\text{ in }\;\mathbb{R}^n.$$ 
Now, if $w$ is a solution of \eqref{minS}, then as done above we get $S(w)=S(u)$ and since the solution of \eqref{minS} belongs to $G$, we get 
$$S_\lambda (u) =S_\lambda (w)=\min\{S_\lambda (u) : u\in A\}.$$
Since $u\in A$, we get $u\in G$. Therefore $u$ solve \eqref{prob} if and only if $u\in G$. 
Thus, $u$ solves \eqref{prob} if and only if it solves \eqref{minS}. Therefore, the problems \eqref{prob} and \eqref{minS} are equivalent and hence the ground state solutions coincide with normalized solutions. 
\end{myproof}
\begin{myproof}{Theorem}{\ref{Theorem 1.5}}
Let $u\in \Gamma$ be a solution of \eqref{prob}, then 
$$\left\|\nabla u\right\|_2^2+\left\|u\right\|_2^2+\lambda[u]^2=\mu\mathcal{A}_p(u)
$$
and
$$\left(\frac{n-2}{2}\right)\left\|\nabla u\right\|_2^2+\frac{n\tau}{2}+\frac{ (n-2s)}{2}\lambda[u]^2=\left(\frac{n+\alpha}{2p}\right)\mu \mathcal{A}_p(u)
.$$
Equating above equations, we get
\begin{eqnarray*}
	\left(\frac{n-2}{2}-\frac{n+\alpha}{2p}\right)\left\|\nabla u\right\|_2^2+\frac{\lambda }{2}\left(n-2s-\frac{n+\alpha}{p}\right)[u]^2& = &
	\left(-\frac{n}{2}+\frac{n+\alpha}{2p}\right)\tau
\end{eqnarray*}
which implies
$$-\left\|\nabla u\right\|_2^2-s\lambda[u]^2 = 0 \;\text{ for }p=\frac{n+\alpha}{n}, 
\text{ that is }
\left\|\nabla u\right\|_2^2 = -s\frac{\lambda [u]^2}{2}<0.$$
Therefore, by contradiction, there does not exist any solution of \eqref{prob} for $p=\frac{n+\alpha}{n}$.\\
Similarly, for $p=\frac{n+\alpha}{n-2}$, we get
$$-\tau=\left(\frac{n-2}{2}-\frac{n-2}{2}\right)\left\|\nabla u\right\|_2^2+\frac{\lambda [u]^2}{2}(n-2s-(n-2))=(1-s)[u]^2>0
.$$
Thus, by contradiction, we conclude also that there is no solution of \eqref{prob} for $p=\frac{n+\alpha}{n-2}$.
\end{myproof}
\section{Appendix}
In this section, we are proving a Pohozaev identity corresponding to our problem type. For that, we follow the ideas in the proof of \cite[Theorem~2.5]{Anthal2025Pohozaev} to handle the mixed operator case together with the proof in \cite[Proposition~3.1]{moroz2013groundstates} to deal with Choquard nonlinearities. Precisely, we state
\begin{theoremA}\label{Pohozaev_identity}
	If $u\in H^1(\mathbb{R}^n)$ is a weak solution of \eqref{prob}, then it satisfies 
	\begin{equation}\label{Pohozaev}
		\left(\frac{n-2}{2}\right)\left\| \nabla u \right\|_2^2+\left(\frac{n-2s}{2}\right)\lambda[u]^2+\frac{n}{2}\left\| u \right\|_2^2=\mu\left(\frac{n+\alpha}{2p}\right)\mathcal{A}_p(u).
	\end{equation}
\end{theoremA}
\begin{proof}
	From \autoref{thm1.2}, we have that $u\in W^{2,q}_{loc}(\mathbb{R}^n)$ for all $q\geq 2$. Therefore by Sobolev embeddings (see for instance \cite[Theorem 2.5.4]{Kesavan2019Topics}), we get $u\in C^{1,\delta}_{loc}(\mathbb{R}^n)$, for all $0<\delta<1$. Following the proof of \cite[Theorem~2.5]{Anthal2025Pohozaev}, since $(I_{\alpha}*|u|^p)\in L^{\infty}(\mathbb{R}^n)$ we get
	\begin{eqnarray}\label{A1}
		\lim_{R\rightarrow 0} L & = & \lim_{R\rightarrow 0}\left(\int_{\mathbb{R}^n} (x\nabla u)(\nabla \psi_R \nabla u)-\frac{1}{2}\int_{\mathbb{R}^n}|\nabla u|^2(x\nabla \psi_R)+\left(\frac{2-n}{2}\right)\int_{\mathbb{R}^n}|\nabla u |^2\psi_R\right.\nonumber\\
		&& -\frac{\lambda C(n,s)}{2.2}\int_{\mathbb{R}^n}\int_{\mathbb{R}^n}\frac{|u(x)-u(y)|^2 div(\Psi_R(x))}{|x-y|^{n+2s}}\nonumber\\
		&&  +\left(\frac{n+2s}{2}\right)\frac{\lambda C(n,s)}{2}\int_{\mathbb{R}^n}\int_{\mathbb{R}^n}\frac{|u(x)-u(y)|^2(x-y)(\Psi_R(x)-\Psi_R(y))}{|x-y|^{n+2s+2}}\nonumber\\
		&&-\frac{n\lambda C(n,s)}{2.2}\int_{\mathbb{R}^n}\int_{\mathbb{R}^n}\frac{|u(x)-u(y)|^2\psi_R(x)}{|x-y|^{n+2s}}\nonumber\\
		&&\left. -\frac{\lambda C(n,s)}{2.2}\int_{\mathbb{R}^n}\int_{\mathbb{R}^n}\frac{|u(x)-u(y)|^2\nabla \psi_R(x)x}{|x-y|^{n+2s}}\right)\nonumber\\
		& = & \left(\frac{2-n}{2}\right)\left\| \nabla u \right\|_2^2+\left(\frac{2s-n}{2}\right)\lambda[u]^2,
	\end{eqnarray}
	where $\psi\in C_c^{\infty}(\mathbb{R}^n)$ is such that $0\leq \psi \leq 1$ with $\psi=1$ for $|x|\leq 1$ and $\psi=0$ for $|x|\geq 2$, for any $R>0$ we further define $\psi_R(x):=\psi(Rx)$, $\Psi_R(x):=\psi(Rx)x$, $\phi(x):= \displaystyle \psi_R\left(\sum_{i=1}^nx_iD_iu(x)\right)$ with $D_i$, the difference quotient operator defined as follows:
	$$D_iv(x)= \frac{v(x+he_i)-v(x)}{h}, \;\;e_i=(0,\cdots , 1, \cdots,0) \text{ being the unit vector along }  x_{i}\text { coordinate, } $$
	and
	$$L=\lim_{h\rightarrow 0 }\left(\int_{\mathbb{R}^n}(\mu(I_{\alpha}*|u|^p)|u|^{p-2}u-u)\phi\right)=\int_{\mathbb{R}^n}(\mu(I_{\alpha}*|u|^p)|u|^{p-2}u-u)\psi_R(x) (x.\nabla u) ,$$
	by \cite[eq~(3.21)]{Anthal2025Pohozaev}. One can notice that in \cite{Anthal2025Pohozaev}, it is assumed that the solution is in $C^{0,\ell}(\mathbb{R}^n)$ for some $\ell>s$. However, the presence of the function with compact support, $\psi$, permits us to generalize everything for $u\in C^{0,\ell}_{loc}(\mathbb{R}^n)$ (and then for $u\in C^{1,\delta}_{loc}(\mathbb{R}^n)$) as well (see also arguments in $B_R(0)\times B_R(0)$ with fixed suitable $R>0$ after (3.3) page 2010  in \cite{ambrosio}).
	Further, by arguing as in \cite[Proposition~3.1]{moroz2013groundstates}, 
    we obtain
	\begin{equation}\label{A2}
		\lim_{R\rightarrow 0} L  =   -\mu\left(\frac{n+\alpha}{2p}\right)\mathcal{A}_p(u)+\frac{n}{2}\left\| u \right\|_2^2.
	\end{equation} 
	Thus, we get \eqref{Pohozaev} by combining \eqref{A1} and \eqref{A2}.
\end{proof}
\noindent \textbf{Acknowledgment:}  We thank the anonymous reviewer who gave very useful suggestions and help us to improve substantially our manuscript. 
\subsection*{Declarations}
\textbf{Ethical Approval} not applicable.\\
\textbf{Funding} This work is supported by the Ministry of Education, Government of India under the Prime Minister's Research Fellows (PMRF) scheme.\\
\textbf{Availability of data and materials} Not applicable.

\end{document}